\documentclass{article}
\usepackage{graphicx} % Required for inserting images
\usepackage{amsthm}
\usepackage{amsmath}
\usepackage{amssymb}
\usepackage{xcolor}
\usepackage{tcolorbox}
\usepackage{tabularx}
\usepackage{physics}
\usepackage{bbold}
\usepackage{hyperref}
\usepackage{proof}
\newtheorem{theorem}{Theorem}
\newtheorem{Lemma}[theorem]{Lemma}
\newtheorem{corollary}[theorem]{Corollary}
\newtheorem{definition}[theorem]{Definition}
\newtheorem{remark}[theorem]{Remark}
\newtheorem{proposition}[theorem]{Proposition}
\usepackage{enumerate}
\usepackage{tikz}
\usetikzlibrary{cd, automata, positioning, arrows}
\newcommand{\T}{\mathbb{T}}
\newcommand{\fcolim}{\text{``}\colim\text{"}}

\newcommand{\Ind}{\operatorname{Ind}}
\newcommand{\Span}{\operatorname{Span}}
\newcommand{\Lex}{\operatorname{Lex}}
\newcommand{\At}{\mathbf{At}}
\newcommand{\Psh}{\mathbf{Psh}}
\newcommand{\Sh}{\mathbf{Sh}}
\newcommand{\Sub}{\mathbf{Sub}}

\newcommand{\yo}{\text{\usefont{U}{min}{m}{n}\symbol{'110}}}
\DeclareFontFamily{U}{min}{}
\DeclareFontShape{U}{min}{m}{n}{<-> dmjhira}{}

\usepackage{xspace} %per def i.e. ed e.g.

\mathchardef\hyphen="2D %trattino corto in matematica

\renewcommand{\epsilon}{\varepsilon}
\renewcommand{\phi}{\varphi}

\newcommand{\Hom}{\operatorname{Hom}}
\newcommand{\colim}{\operatorname{colim}}
\newcommand{\dom}{{\operatorname{dom}}}

\newcommand{\id}{\operatorname{id}}

\newcommand{\bigAND}{\bigwedge}

\newcommand{\bigOR}{\bigvee}

\newcommand{\Cat}{\mathbf{Cat}}
\newcommand{\Set}{\mathbf{Set}}
\newcommand{\Mod}{{\mathbf{Mod}}}

\newcommand{\dcat}{\mathbb{D}}

\newcommand{\pcat}{\mathbb{P}}

\makeatletter
\newbox\xrat@below
\newbox\xrat@above
\newcommand{\xrightarrowtail}[2][]{%
	\setbox\xrat@below=\hbox{\ensuremath{\scriptstyle #1}}%
	\setbox\xrat@above=\hbox{\ensuremath{\scriptstyle #2}}%
	\pgfmathsetlengthmacro{\xrat@len}{max(\wd\xrat@below,\wd\xrat@above)+.6em}%
	\mathrel{\tikz [>->,baseline=-.75ex]
		\draw (0,0) -- node[below=-2pt] {\box\xrat@below}
		node[above=-2pt] {\box\xrat@above}
		(\xrat@len,0) ;}}
\makeatother

\title{De Morgan's law in toposes I}
\author{Olivia Caramello \and Yorgo Chamoun}

\begin{document}

\maketitle
\begin{abstract}
    We study toposes satisfying De Morgan's law, in particular we give characterizations of geometric theories whose classifying topos is De Morgan, clarifying the link with the amalgamation property of the category of models of such theory. We then give several ways of turning a topos into a De Morgan topos.
\end{abstract}

\tableofcontents

\section{Introduction}

It is now well-known that every geometric theory has a classifying topos (see for instance \cite{caramello2018theories}). In parallel, every topos has an internal language which is (at least) intuitionistic, and may or may not satisfy some axioms like the excluded-middle or De Morgan's law. Thus, there is natural notion of De Morgan theory, namely a geometric theory whose classifying topos satisfies De Morgan's law. In this paper, we undertake a systematic study of these theories, from a model-theoretic and topos-theoretic point of view.\\
\\
The first question to ask is whether there exist syntactic and semantic criteria for a theory to be De Morgan. Then we can think of a (more or less canonical) way to turn a theory into a De Morgan theory, or equivalently a topos into a De Morgan topos. The question of a syntactic criterion was answered by Bagchi in \cite{bagchi1992morgan} for coherent theories, and in full generality by Caramello in \cite{caramello2009morgan}. A semantic criterion for a coherent version of De Morgan theories is given in \cite{harun1996applications} without proof and attributed to Joyal and Reyes. Bagchi also gives a semantic characterization for coherent theories in \cite{bagchi1992morgan}, involving finitely presented structures. Both involve the amalgamation property, which also appears in the characterization of De Morgan presheaf toposes. The amalgamation property thus appears to be very closely linked to De Morgan's law. Finally the question of enforcing De Morgan's law has been answered in two orthogonal ways by Johnstone in \cite{johnstone1980gleason} and Caramello in \cite{caramello2009morgan}. If these questions seem too `abstract', one can take a look at the results in \cite{caramello2009morgan2}, where for instance it is proved that the `DeMorganization' of the (coherent) theory of fields is the
(geometric) theory of fields of finite characteristic which are algebraic over their prime fields. This shows that topos-theoretic features of logic, even the most abstract ones, can indeed have concrete applications for the working model theorist. Other interesting applications can be found for instance in \cite{harun1996applications}. However, the answers to the precedent questions were not completely satisfactory, and other more interesting questions arose, which gives this work its \emph{raison d'être}. \\
\\
In section \ref{sec:dml}, after recalling some basic facts about De Morgan's law, we define De Morgan toposes and theories and give basic results about them. In particular, we generalize the argument of Bagchi to give a syntactic criterion for a theory to be De Morgan. Then in section \ref{sec:presheaf_type}, we revisit the well-known characterization of De Morgan's law for theories of presheaf type, proving that amalgamation for finitely presentable models is equivalent to amalgamation for all set-based models. In section \ref{sec:model_amalg}, we investigate the semantic criterion of Joyal and Reyes, give a proof, interpret it topos-theoretically, and propose a conceptual point of view on the matter. We also provide a categorical interpretation of the semantic characterization of Bagchi in the particular case of presheaf-type theories. In section \ref{sec:enforce}, we use the Gleason cover of a topos, introduced by Johnstone in \cite{johnstone1980gleason}, to enforce De Morgan's law. More specifically, we use new tools to explicitly describe the Gleason cover of a topos, and reprove the main theorems of Johnstone in the language for relative topos theory introduced in \cite{caramello2021relative}. We also investigate some further applications, once again related to amalgamation. Lastly, we introduce another way of enforcing De Morgan's law, which we plan to further investigate in future work.

\section{De Morgan's law in toposes}\label{sec:dml}
\subsection{De Morgan's law}
The two well-known De Morgan laws can be written as:
$$\neg (x\lor y) = \neg x\land \neg y$$
$$\neg (x\land y) = \neg x\lor \neg y$$
These are true in any Boolean algebra, i.e. in classical logic, but not in any Heyting algebra, i.e. in intuitionistic logic, which is the logic of elementary toposes. Let us focus on these algebras for now. The results of this section are well-known, we give some proofs for the reader's convenience. First, note that De Morgan laws are \emph{almost} true in any Heyting algebra.

\begin{proposition}\label{true_dml}
    In any Heyting algebra $H$, for any two elements $x,\,y\in H$, we have
    \begin{align}
        \neg (x\lor y) = \neg x\land \neg y\\
        \neg (x\land y) \geq \neg x\lor \neg y
    \end{align}
    
\end{proposition}
\begin{proof}
    Using that $z\leq t$ implies $\neg t\leq \neg z$ (because $z\leq t$ implies $z\leq z\land t$ implies $z\land \neg t\leq 0$ implies $\neg t\leq \neg z)$ and the universal properties of $\lor$ and $\land$, the only non trivial inequality is $\neg x\land \neg y\leq \neg (x\lor y)$. But this is equivalent to $\neg x\land \neg y \land (x\lor y)=0$, and we conclude using the distributivity law:
    $\neg x\land \neg y \land (x\lor y)\leq \neg x\land (x\lor y)\leq \neg x\land y$ and $\neg x\land \neg y \land (x\lor y)\leq \neg y$ so $\neg x\land \neg y \land (x\lor y)\leq \neg x\land y \land \neg y=0$. 
\end{proof}

\begin{definition}
    A Heyting algebra is a De Morgan algebra or Stone algebra if it satisfies $\neg (x\land y) \leq \neg x\lor \neg y$. 
\end{definition}
In general this inequality is what will be called the De Morgan's law. The following law will also be called De Morgan's law.

\begin{proposition}\label{comp_join}
    A Heyting algebra $H$ is De Morgan if and only if it satisfies $\neg x\lor \neg\neg x=1$ for all $x\in H$.
\end{proposition}
\begin{proof}
    In one direction, just apply De Morgan's law for $x$ and $\neg x$. In the other, first we have to see that $\neg\neg$ preserves meets and joins. \\
    \\
    We claim that $\neg\neg(x\land y)=\neg\neg x\land \neg\neg y$ in any Heyting algebra. Indeed, one direction is trivial, for the other apply (1) of Proposition~\ref{true_dml} to $\neg x$ and $\neg y$: $\neg (\neg x\lor \neg y)=\neg\neg x\land \neg \neg y$, then use (2).\\
    \\
    We claim that $\neg\neg(x\lor y)=\neg\neg x\lor \neg\neg y$ in a Heyting algebra satisfying $\neg x\lor \neg\neg x=1$ for all $x$. Indeed, noticing that by applying (1) of Proposition~\ref{true_dml} to $\neg\neg x$ and $\neg\neg y$ we have that $\neg\neg(x\lor y)=\neg\neg(\neg\neg x\lor \neg\neg y)$, we just have to prove that the join of complemented elements is complemented, because in such an algebra complemented elements coincide with $\neg\neg$-closed elements. This follows from:
\begin{align*}
    ((x\lor\neg x)\land y)\lor ((x\lor\neg x)\land \neg y) & = (x\land y)\lor(\neg x\land y)\lor(x\land \neg y)\lor(\neg x\land \neg y)\\
    & \leq x\lor y\lor  (\neg x\land \neg y)\\
    & =(x\lor y)\lor\neg (x\lor y)
\end{align*}
This shows that if $x$ and $y$ are complemented, i.e. $x\lor\neg x=1$ and $y\lor\neg y=1$, then their join also is.\\
\\
Now putting everything together, we have 
\begin{align*}
    \neg(x\land y) & =\neg\neg\neg(x\land y)\\
    & = \neg(\neg\neg x \land \neg\neg y)\\
    & = \neg\neg(\neg x\lor \neg y)\\
    & = \neg\neg\neg x\lor \neg\neg\neg y\\
    & = \neg x\lor \neg y
\end{align*}
\end{proof}

We also give a Lemma~that will be useful later.

\begin{Lemma}\label{Hleqx}
    Let $H$ be a Heyting algebra. For an element $x\in H$, we write $H_{\leq x}$ for the Heyting algebra of elements of $H$ below $x$. Then $H$ is a Stone algebra if and only if $H_{\leq x}$ is a Stone algebra for all $x\in H$.
\end{Lemma}
\begin{proof}
    One direction is trivial. For the other, suppose that $H$ is a Stone algebra, and consider $H_{\leq x}$ for some $x$. Note that the Heyting implication of $a,\,b$ in $H_{\leq x}$ is given by $(a\rightarrow b)\land x$, so the double pseudo-complement of $a$ is given by $x\land \neg(x\land\neg a)$. Now for any $b\leq x$, we have 
\begin{align*}
    b\leq x\land \neg(x\land\neg a) & \iff b\leq \neg(x\land\neg a)\\
    & \iff b\land x \land \neg a=0\\
    & \iff b\leq\neg\neg a\\
    & \iff b\leq\neg\neg a\land x
\end{align*}
so by Yoneda, the double negation of $a\in H_{\leq x}$ is given by $\neg\neg a\land x$. Now De Morgan's law for an element $a$ is given by $(\neg a\land x)\lor(\neg\neg a \land x)=x$. But $(\neg a\land x)\lor(\neg\neg a \land x)=x\land(\neg a\lor\neg\neg a)=x\land 1=x$.
\end{proof}

\subsection{De Morgan toposes}

Now we want to define what it means for an elementary topos to satisfy De Morgan's law. It is well-known that the internal logic of an elementary topos is (at least) intuitionistic, so satisfying De Morgan's law should be equivalent to the fact that De Morgan's law holds in the internal logic of the topos. However we will avoid going into the internal logic, and we will define a De Morgan topos in an `external way'. 

\begin{definition}
    A topos $\mathcal{E}$ is said to be De Morgan if its subobject classifier is an internal Stone algebra.
\end{definition}

There are a number of equivalent formulations of this property, some of which are listed in \cite{johnstone2006conditions}. We review below those which are most relevant for us. First, recall that there is a `canonical' Boolean algebra in every topos, given by the equalizer 
\begin{center}
    \begin{tikzcd}
\Omega_{\neg\neg} \arrow[r, tail]
& \Omega \arrow[r, shift left, "\neg\neg"]
\arrow[r, shift right, "\id"']
& \Omega 
\end{tikzcd}
\end{center}
Since equalizers are computed pointwise, this is just given at an object $c$ by the subset of $\Omega(c)$ of the $\neg\neg$-closed elements. The (pointwise) join operation (at $c$) is given by $(x,y)\mapsto \neg\neg (x\lor y)$ where $\neg$ and $\lor$ refer to the operations in $\Omega(c)$, and all the other operations are inherited from $\Omega(c)$, so it is easy to check that it is indeed a Boolean algebra pointwise. Now, $\neg\top=\bot$ and $\neg\bot=\top$, $1\xrightarrow{\top}\Omega$ and $1\xrightarrow{\bot}\Omega$ factor through $\Omega_{\neg\neg}$, which gives a morphism $(\top,\bot):1\sqcup 1\xrightarrow{}\Omega_{\neg\neg}$. In fact, it is a monomorphism, because clearly $\top$ and $\bot$ are disjoint subobjects of $\Omega$ and so of $\Omega_{\neg\neg}$, and in a topos the coproduct of disjoint subobjects is their join (Proposition~IV.7.6 \cite{maclane2012sheaves}).

\begin{proposition}[\cite{johnstone2006conditions}]
    For a topos $\mathcal{E}=\mathbf{Sh}(\mathcal{C},J)$, the following are equivalent:
    \begin{enumerate}[(i)]
        \item $\mathcal{E}$ is De Morgan;
        \item For every object $E$ of $\mathcal{E}$, $\textup{Sub}_\mathcal{E}(E)$ is a Stone algebra, i.e. for any subobject $A$, $\neg A\lor\neg\neg A=E$ the maximal suboject;
        \item The subobject $\bot:1\xrightarrowtail{\enspace}\Omega$ has a complement;
        \item A subobject is $\neg\neg$-closed if and only if it has a complement;
        \item The monomorphism $(\top,\bot):1\sqcup 1\xrightarrow{}\Omega_{\neg\neg}$ is an isomorphism.
    \end{enumerate}
\end{proposition}

Note that we can also define a \emph{Boolean topos} to be a topos such that $\Omega$ is an internal Boolean algebra, equivalently every $\textup{Sub}_\mathcal{E}(E)$ is a Boolean algebra. Obviously, any Boolean topos is in particular a De Morgan topos. In fact, for any topos $\mathcal{E}$, there exists a Boolean topos and a geometric embedding $\mathbf{sh}_{\neg\neg}(\mathcal{E})\hookrightarrow{}\mathcal{E}$, called the Booleanization of $\mathcal{E}$, which can be characterized by the property of being the unique dense Boolean subtopos of $\mathcal{E}$, where a subtopos is \emph{dense} if it contains the initial object of $\mathcal{E}$. The subobject classifier of $\mathbf{sh}_{\neg\neg}(\mathcal{E})$ will be $\Omega_{\neg\neg}$ as defined above. We also mention that we have a similar result for De Morgan's law, i.e. a subtopos and a geometric embedding $\mathbf{sh}_m(\mathcal{E})\hookrightarrow{}\mathcal{E}$ such that $\mathbf{sh}_m(\mathcal{E})$ is De Morgan, called the DeMorganization, and with subobject classifier $\Omega_m$ an internal Stone algebra of $\mathcal{E}$. The interested reader can look at \cite{caramello2009morgan}.\\
\\
Finally, let us also give an easy yet important definition.

\begin{definition}
\begin{enumerate}[(i)]
    \item A geometric theory $\T$ is De Morgan if its classifying topos is De Morgan.
    \item A geometric theory $\T$ is Boolean if its classifying topos is Boolean.
\end{enumerate}
     
\end{definition}
One can wonder if this is equivalent to some syntactic property of the theory. This is the subject of the next section. Remark~that since every topos is the classifying topos of some geometric theory, this will give another equivalent condition for a topos to be De Morgan. 

\subsection{A syntactic criterion}
In this section we generalize a syntactic criterion for a theory to be De Morgan, from the coherent to the geometric case. Bagchi characterized in \cite{bagchi1992morgan} being De Morgan for a coherent theory by means of its definable subobjects. More precisely, given a coherent theory $\mathbb{T}$, define $D_n(\mathbb{T})$ to be the lattice of infinite disjunctions of coherent formulas in a context of length $n$, quotiented by the relation $\phi\sim\psi$ if and only if $\T\text{-}\Mod(\Set)\models (\phi\Leftrightarrow\psi)$, with the corresponding order relation $\phi\leq\psi$ if and only if $\T\text{-}\Mod(\Set)\models (\phi\Rightarrow\psi)$, and operations inherited from the logical ones. Then we gets the following result.
\begin{theorem}[\cite{bagchi1992morgan}]\label{bagchi}
    A coherent theory $\mathbb{T}$ is De Morgan if and only if $D_n(\mathbb{T})$ is a De Morgan algebra for every $n\in\mathbb{N}$.
\end{theorem}
\begin{proof}[Idea of proof]
    There is an isomorphism between $(D_n(\mathbb{T}))_{\leq\phi}$ and $\Omega(\{x.\phi\})$, where $\Omega$ is the subobject classifier of the (coherent) classifying topos of $\mathbb{T}$. Then just use (a simple variation of) Lemma~\ref{Hleqx}. 
\end{proof}
We will give a full proof as variation of our generalization. What is interesting in this Theorem~is that it generalizes directly to all the `higher De Morgan laws', i.e. the Lie properties.
\begin{definition}
    A Heyting algebra is said to satisfy the Lee property $I_r$ for $r\geq 1$ if it satisfies
    $$\bigwedge_{0\leq i<j\leq r}\big( x_i\land x_j=0\big) \xrightarrow{} \big(\bigvee_i \neg x_i = 1\big)$$
\end{definition}
Notice that De Morgan's law is equivalent to $I_1$. We can also define what it means for a theory to satisfy $I_r$ like we did for the De Morgan property. \\
\\
Now Bagchi's Theorem~\ref{bagchi} is still true if we replace ``De Morgan" by $I_r$ for any $r\geq 1$. We want to generalize this Theorem~for geometric theories, and to do so it will be more convenient to aim for a purely syntactic criterion. Indeed, Bagchi considers finitary theories, but he still needs to consider infinite disjunctions of formulas because they naturally arise in the classifying topos, so a syntactic criterion is less natural in his case, though possible (Remark~\ref{completeness}). He also needs to use the completeness Theorem~to do the link between semantics and syntax, but this is not true for general geometric theories so his proof does not generalize directly. On the other hand, the `syntactic version' of Theorem~\ref{bagchi} is true for geometric theories. It basically says that a theory $\mathbb{T}$ is De Morgan if and only if the $\textup{Sub}_{\mathcal{C}_\mathbb{T}}(\{x.\top\})$ are De Morgan, where $\mathcal{C}_\T$ is the syntactic category of $\T$ (see for example Definition 1.4.1 of \cite{caramello2018theories}), which is a direct consequence of the results in \cite{caramello2009morgan}. We give a more direct proof, inspired by Bagchi's methodology, which has the advantage of generalizing directly to all the $I_r$.
\begin{theorem}[\cite{caramello2009morgan} Theorem~2.11]\label{dmltheory}
    A geometric theory $\mathbb{T}$ is De Morgan if and only if for every context $x$, the Heyting algebra $\textup{Sub}_{\mathcal{C}_\mathbb{T}}(\{x.\top\})$ is a De Morgan algebra.
\end{theorem}
\begin{proof}
    First, we show that for any object $d$ of the syntactic category $\mathcal{C}_\T$, $\Omega(d)$ is isomorphic to $\textup{Idl}_c(\textup{Sub}_{\mathcal{C}_\mathbb{T}}(d))$, where $\textup{Idl}_c$ are the ideals which are closed under arbitrary disjunction. Indeed, Consider $S\in\Omega(d)$, it is a $J^{geom}$-closed sieve, where $J^{geom}$ is the geometric topology on $\mathcal{C}_\T$. Using the cover-mono factorization, we easily see that $S$ is generated by its monomorphisms: given $f\in S$, its monic part $i$ is in $S$, because $i^*(S)$ is a $J^{geom}$-covering sieve, because it contains the cover part of $f$. By a similar argument, we see that the monics of $S$ are closed under arbitrary disjunction, and the fact that $S$ is a sieve gives that they form an ideal. Conversely, given an ideal $I\in \textup{Idl}_c(\textup{Sub}_{\mathcal{C}_\mathbb{T}}(d))$, we take the sieve it generates, noted $(I)$, then we take its $J^{geom}$-closure, i.e. the sieve of morphisms it covers. To check that we get an isomorphism, we just need to check that this operation does not add monics to $I$. It is clear for the ``generated sieve" part, because $I$ is an ideal. Now if $f:c\xrightarrow{}d$ is a monic such that $f^*((I))$ is covering, we get a family $(f_i:c_i\xrightarrow{}c)_i$ such that $\bigvee_{i}Im(f_i)=c$ by the definition of the geometric topology, and such that each $f\circ f_i$ factors through some $g_i$ a monomorphism in $I$. But by uniqueness of the cover-mono factorization, we have that $Im(f\circ f_i)=f\circ Im(f_i)$, and $Im(f\circ f_i)\leq g_i$ by the universal property of the image. Since $I$ is an ideal, this means that $f\circ Im(f_i)\in I$ for all $i$, so $f$ is in $I$ because it is closed under arbitrary unions.\\
    \\
    Now it is easy to conclude like in the proof of Theorem~\ref{bagchi}, because ideals that are closed under arbitrary disjunction are just principal ideals, so we get that $\Omega(d)\cong \textup{Sub}_{\mathcal{C}_\mathbb{T}}(d)=\textup{Sub}_{\mathcal{C}_\mathbb{T}}(\{x.\top\})_{\leq d}$ for $d$ a formula in a context $x$.
\end{proof}

    We clearly see that the proof is still valid if we replace $I_1$ by an arbitrary $I_r$: 
\begin{theorem}
    A geometric theory $\mathbb{T}$ satisfies Lie property $I_r$ for $r\geq 1$ if and only if for every context $x$, $\textup{Sub}_{\mathcal{C}_\mathbb{T}}(\{x.\top\})$ satisfies $I_r$. In addition, $\T$ is Boolean if and only if the $\textup{Sub}_{\mathcal{C}_\mathbb{T}}(\{x.\top\})$ are Boolean.
\end{theorem}

\begin{remark}\label{completeness}
    If $\mathbb{T}$ is coherent, we get an isomorphism between $\textup{Sub}_{\mathcal{C}_\mathbb{T}}(\{x.\top\})$ and $D_n(\mathbb{T})$. This is just a consequence of the classical completeness Theorem~for coherent logic. Indeed, if we write $E_n(\mathbb{T})$ for the restriction of $D_n(\mathbb{T})$ to coherent formulas, we have that $D_n(\T)=\textup{\textup{Idl}}(E_n(\T))$ as noticed in \cite{bagchi1992morgan}. On the other hand, considering the coherent syntactic category $\mathcal{C}^{coh}_{\mathbb{T}}$ of $\mathbb{T}$, mimicking the last proof we get that $\Omega^{\mathbf{Sh}(\mathcal{C}^{coh}_\T,J^{coh})}(d)=\textup{Idl}(\textup{Sub}_{\mathcal{C}^{coh}_{\mathbb{T}}}(d))$. We conclude by $\textup{Sub}_{\mathcal{C}^{coh}_{\mathbb{T}}}(d)\cong (E_n)_{\leq d}$ (by the completeness Theorem~for coherent logic and Lemma~D1.4.4 of \cite{johnstone2002sketches}), and the fact that $$\Omega^{\mathbf{Sh}(\mathcal{C}^{coh}_\T,J^{coh})}(\{x.\phi\})\cong\Omega^{\mathbf{Sh}(\mathcal{C}_\T,J^{geom})}(\{x.\phi\})$$ for every coherent $\phi$ (because the equivalence $\mathbf{Sh}(\mathcal{C}_\T,J^{geom})\cong \mathbf{Sh}(\mathcal{C}_\T^{coh},J^{coh})$ is induced by the obvious inclusion functor $\mathcal{C}^{coh}_\T\xrightarrow{}\mathcal{C}_\T$):
    \begin{align*}
        \textup{Sub}_{\mathcal{C}_\mathbb{T}}(\{x.\top\}) & =\Omega^{\mathbf{Sh}(\mathcal{C}_\T,J^{geom})}(\{x.\top\})\\
        & =\Omega^{\mathbf{Sh}(\mathcal{C}^{coh}_\T,J^{coh})}(\{x.\top\})\\
        & =\textup{Idl}(\textup{Sub}_{\mathcal{C}^{coh}_{\mathbb{T}}}(\{x.\top\}))\\
        & =\textup{Idl}(E_n(\T))\\
        & =D_n(\T)
    \end{align*}
    Note that this gives a proof of the isomorphism in the proof of Theorem~\ref{bagchi}.
\end{remark}

\subsection{De Morgan's law on presheaf toposes}\label{sec:presheaf_type}
Recall that a theory of presheaf type is a theory whose classifying topos is equivalent to a presheaf topos. More precisely:

\begin{definition}
     A finitely presented model of a geometric theory $\T$ is a (set-based) model $M$ such that there exists a geometric formula $\phi$ and a set of parameters, called set of generators $\xi$ such that $M\models\phi(\xi)$ and for each tuple of parameters of the appropriate length $\xi'$ in a model $N$ such that $N\models\phi(\xi')$, there exists a unique model homomorphism $h:M\xrightarrow{}N$ such that $h(\xi)=\xi'$.
 \end{definition}

We note $f.p.\T\text{-}\Mod(\Set)$ the full subcategory of $\T\text{-}\Mod(\Set)$ of the finitely presented models of $\T$. For a theory of presheaf type $\T$, we have an equivalence $\mathbf{Sh}(\mathcal{C}_\T,J^{geom})\cong \mathbf{Psh}((f.p.\T\text{-}\Mod(\Set))^{op})$. Thus, theories of presheaf type are typically theories where syntax and semantics are well-behaved with respect to each other. In particular, it is well-known that the property of being a De Morgan topos can be characterized semantically: 

 \begin{proposition}\label{Ore}
     A topos $\mathbf{Psh}(\mathcal{C})$ is De Morgan if and only if $\mathcal{C}$ satisfies the right Ore condition, i.e. if $\mathcal{C}^{op}$ satisfies amalgamation.
 \end{proposition}
 
In this section, we prove that a category has amalgamation if and only if its ind-completion has amalgamation. As a direct corollary, we get that a theory of presheaf type is De Morgan if and only if its set-based models have amalgamation.\\
\\
We will have to do a transfinite recursion, so we need some lemmas which will help us deal with size issues. We fix a category $\mathcal{C}$. Note that by Theorem~1.5 of \cite{rosicky1994locally}, we can assume that the formal colimits in $\Ind(\mathcal{C})$ are indexed by directed posets, as long as the operations we apply on them respect their directedness. So we will assume it in what follows. For a formal colimit $\fcolim_{i\in I}A_i$, the canonical functor $A_i\xrightarrow{}\fcolim_{i\in I}A_i$ will be noted $\iota_i$ or $\iota_{A_i}$. For $I'\subseteq I$, the canonical functor $\fcolim_{i\in I'}A_i\xrightarrow{}\fcolim_{i\in I}A_i$ will be noted $\iota_{I'}$ or $\iota_{I'\subseteq I}$.

\begin{Lemma}
    Let $I$ be a directed poset, $J\subseteq I$ any subset. There is a directed subposet $J\subseteq J^*\subseteq I$ such that $|J^*|=|J|$ if $|J|\geq\aleph_0$, and $|J^*|<\aleph_0$ if $|J|<\aleph_0$.
\end{Lemma}
\begin{proof}
    By the proof of Lemma~1.6 of \cite{rosicky1994locally}.
\end{proof}

This operation will be called closure and will always be noted $(-)^*$.

\begin{Lemma}\label{eq_size}
    Let $(B_j)_{j\in J}$ be a directed digram, $J'\subseteq J$, $f,g:\fcolim_{i\in I}A_i\xrightarrow{}\fcolim_{j\in J'}B_j$ two morphisms in $\Ind(\mathcal{C})$. If $\iota_{J'}\circ f=\iota_{J'}\circ g$, then there exists $J'\subseteq J''\subseteq J$ with $|J''|\leq max\{|I|,|J'|\}$ if one of them is infinite, and $|J''|$ is finite if both $I$ and $J'$ are, such that $\iota_{J'\subseteq J''}\circ f=\iota_{J'\subseteq J''}\circ g$.
\end{Lemma}
\begin{proof}
    The idea is that we only need to check equality on each $A_i$. Since $A_i$ is finitely presentable in $\Ind(\mathcal{C})$, there are $B_{j_f}$ and $B_{j_g}$, $j_f,j_g\in J'$ such that $f\circ\iota_i$ and $g\circ\iota_i$ factor through them respectively. The situation is the following:
% https://q.uiver.app/#q=WzAsNixbMCwxLCJBX2kiXSxbMSwwLCJCX3tqX2Z9Il0sWzEsMiwiQl97al9nfSJdLFsxLDEsIlxcZmNvbGltX3tpXFxpbiBJfUFfaSJdLFszLDEsIlxcZmNvbGltX3tqXFxpbiBKJ31CX2oiXSxbNCwxLCJcXGZjb2xpbV97alxcaW4gSn1CX2oiXSxbMCwxLCJmJyJdLFswLDIsImcnIiwyXSxbMyw0LCJnIiwyLHsib2Zmc2V0IjoxfV0sWzQsNV0sWzEsNF0sWzIsNF0sWzAsM10sWzMsNCwiZiIsMCx7Im9mZnNldCI6LTF9XV0=
\[\begin{tikzcd}
	& {B_{j_f}} \\
	{A_i} & {\fcolim_{i\in I}A_i} && {\fcolim_{j\in J'}B_j} & {\fcolim_{j\in J}B_j} \\
	& {B_{j_g}}
	\arrow[from=1-2, to=2-4]
	\arrow["{f'}", from=2-1, to=1-2]
	\arrow[from=2-1, to=2-2]
	\arrow["{g'}"', from=2-1, to=3-2]
	\arrow["g"', shift right, from=2-2, to=2-4]
	\arrow["f", shift left, from=2-2, to=2-4]
	\arrow[from=2-4, to=2-5]
	\arrow[from=3-2, to=2-4]
\end{tikzcd}\]
where both composites are equal because $\iota_{J'}\circ f\circ \iota_i=\iota_{J'}\circ g\circ \iota_i$. Since $A_i$ is finitely presentable, we have that $\phi:\Hom(A_i,\fcolim_{j\in J}B_j)\cong\colim_{j\in J}(A_i,B_j)$. Since $\phi^{-1}(f')=\phi^{-1}(g')$, we have that $[f']=[g']$ in $\colim_{j\in J}(A_i,B_j)$, which means that there are $u(i)\in J$, $h_f:B_{j_f}\xrightarrow{}B_{u(i)}$ and $h_g:B_{j_g}\xrightarrow{}B_{u(i)}$ (belonging to the diagram) such that $h_f\circ f'=h_g\circ g'$. So if we had that $u(i)\in J'$, we would have $f\circ\iota_i=g\circ\iota_i$. Now we just conclude by defining $J'':=(J'\cup\{u(i)\,|\,i\in I\})^*$.
\end{proof}

\begin{Lemma}\label{facto_size}
    Let $f:\fcolim_{i\in I}A_i\xrightarrow{}\fcolim_{j\in J}B_j$. Then there exists $J'\subseteq J$ with $|J'|\leq|I|$, such that $f$ factors through $\iota_{J'}$.
\end{Lemma}
\begin{proof}
    We prove this by transfinite recursion on the size of $I$. First, we handle the case where $I$ is finite. Notice that a finite directed poset has a maximum, so a finite formal directed colimit is isomorphic to an object of $\mathcal{C}$. So we just need to use that the objects of $\mathcal{C}$ are finitely presentable in $\Ind(\mathcal{C})$.\\
    \\
    Now suppose that $|I|=\lambda$ is an infinite cardinal. By Lemma~1.6 of \cite{rosicky1994locally}, we can find an increasing sequence $(I_\beta)_{\beta<\lambda}$ indexed by ordinals smaller than $\lambda$ which is continuous in the sense that $I_\beta=\bigcup_{\beta'<\beta}I_{\beta'}$ for every limit ordinal $\beta$, and such that $|I_\beta|<\lambda$ for all $\beta$ and $I=\bigcup_{\beta<\lambda}I_\beta$. Now we recursively build a continuous increasing sequence $(J_\beta)_{\beta<\lambda}$ such that $|J_\beta|\leq max\{|I_\beta|,\aleph_0\}$ for all $\beta<\lambda$ and a morphism of chains 
% https://q.uiver.app/#q=WzAsNixbMCwwLCJcXGZjb2xpbV97alxcaW4gSl8wfUJfaiJdLFsxLDAsIlxcZmNvbGltX3tqXFxpbiBKXzF9Ql9qIl0sWzIsMCwiXFxkb3RzIl0sWzAsMSwiXFxmY29saW1fe2lcXGluIElfMH1BX2kiXSxbMSwxLCJcXGZjb2xpbV97aVxcaW4gSV8xfUFfaSJdLFsyLDEsIlxcZG90cyJdLFswLDFdLFsxLDJdLFszLDAsImZfMCJdLFszLDRdLFs0LDEsImZfMSJdLFs0LDVdXQ==
\[\begin{tikzcd}
	{\fcolim_{j\in J_0}B_j} & {\fcolim_{j\in J_1}B_j} & \dots \\
	{\fcolim_{i\in I_0}A_i} & {\fcolim_{i\in I_1}A_i} & \dots
	\arrow[from=1-1, to=1-2]
	\arrow[from=1-2, to=1-3]
	\arrow["{f_0}", from=2-1, to=1-1]
	\arrow[from=2-1, to=2-2]
	\arrow["{f_1}", from=2-2, to=1-2]
	\arrow[from=2-2, to=2-3]
\end{tikzcd}\]
    where the horizontal morphisms are the canonical ones, such that $f\circ\iota_{I_\beta}=\iota_{J_\beta}\circ f_\beta$ for all $\beta$. To build $J_0$ and $f_0$, just use the induction hypothesis on $f\circ\iota_{I_0}$. At successor ordinal $\beta=\beta'+1$, first we use the induction hypothesis on $f\circ\iota_{I_{\beta}}$ to get $J'_{\beta}$ and $f'$ such that $f\circ\iota_{I_{\beta'}}=\iota_{J'}\circ f'$. We can assume that $J_\beta\subseteq J'_{\beta}$, because we can always take the closure of their union if it is not the case. Now by the induction hypothesis of the construction, we have $f\circ\iota_{I_{\beta'}}=\iota_{J_{\beta'}}\circ f_{\beta'}$, so $$\iota_{J'_{\beta}}\circ\iota_{J_{\beta'}\subseteq J'_{\beta}}\circ f_{\beta'}=f\circ\iota_{I_\beta}\circ\iota_{I_{\beta'}\subseteq I_\beta}=\iota_{J'_{\beta}}\circ f'\circ\iota_{I_{\beta'}\subseteq I_\beta}$$
    Now by Lemma~\ref{eq_size} there is $J_\beta\supseteq J'_{\beta}$ having the good size such that $\iota_{J'_{\beta}\subseteq J_\beta}\circ\iota_{J_{\beta'}\subseteq J'_{\beta}}\circ f_{\beta'}=\iota_{J'_{\beta}\subseteq J_\beta}\circ f'\circ\iota_{I_{\beta'}\subseteq I_\beta}$, and we just define $f_\beta:=\iota_{J'_{\beta}\subseteq J_\beta}\circ f'$. At limit ordinal $\beta$, we just take the union of the $J_{\beta'}$ for $\beta'<\beta$. Notice that then $\fcolim_{j\in J_\beta}B_j$ is the colimit of the chain which is already constructed, and this is also true for $\fcolim_{i\in I_\beta}A_i$, so $f_\beta$ is given by functoriality of the colimit, and we get that $f\circ\iota_{I_\beta}=\iota_{J_\beta}\circ f_\beta$ by the universal property of the colimit. This finishes the construction. But now it suffices to take $J':=\bigcup_{\beta<\lambda}J_\beta$. This is again the colimit of the diagram, which gives the map through which $f$ factors.
\end{proof}

Note that we do not really need transfinite induction in the last proof, since we can essentially generalize the argument of the proof of Lemma~\ref{eq_size}, but we will need this construction in the proof of the main theorem, which we are now ready to do. A diagram of the form
\begin{center}
    \begin{tikzcd}
        A & & C\\
        & B\arrow[ul,"g"]\arrow[ur,"f"'] &
    \end{tikzcd}
\end{center}
will be noted $\Span(f,g)$.
\begin{theorem}
    For a category $\mathcal{C}$, $\mathcal{C}$ has the amalgamation property if and only if $\Ind(\mathcal{C})$ has the amalgamation property.
\end{theorem}
\begin{proof}
    Suppose that $\Ind(\mathcal{C})$ has the amalgamation property, and take a diagram $\Span(f,g)$ in $\mathcal{C}$. This diagram admits an amalgamation in $\Ind(\mathcal{C})$:
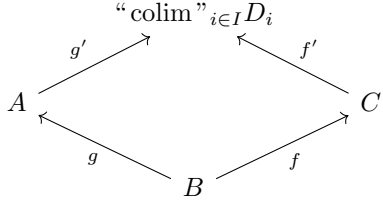
    \begin{center}
    \begin{tikzcd}
        & \fcolim_{i\in I}D_i & \\
        A\arrow[ur, "g'"] & & C\arrow[ul, "f'"']\\
        & B\arrow[ul,"g"]\arrow[ur,"f"'] &
    \end{tikzcd}
\end{center}
$g'$ factors through $D_t$ and $f'$ factors through $D_z$: $g'=\iota_t\circ g''$, $f'=\iota_z\circ f''$. Since $\Hom(B,\fcolim_{i\in I}D_i)\cong\colim_{i\in I}\Hom(B,D_i)$, and since $g''\circ g$ and $f''\circ f$ have the same image by this isomorphism, then $[g''\circ g]=[f''\circ f]$, which means that there is an amalgamation of $\Span(f,g)$ in $\mathcal{C}$.\\
\\
Conversely, suppose that $\mathcal{C}$ has the amalgamation property. We prove by transfinite induction on $\lambda$ that any diagram of the form
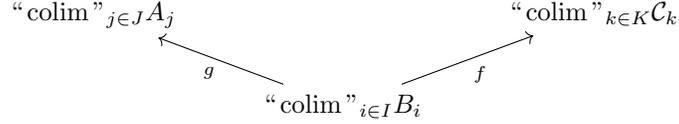
\begin{center}
    \begin{tikzcd}
        \fcolim_{j\in J}A_j & & \fcolim_{k\in K}\mathcal{C}_k\\
        & \fcolim_{i\in I}B_i\arrow[ul,"g"]\arrow[ur,"f"'] &
    \end{tikzcd}
\end{center}
with $|I|,|J|,|K|\leq\lambda$ can be amalgamated. If all three posets are finite, this reduces to amalgamation in $\mathcal{C}$. Now suppose $\lambda\geq\aleph_0$, and that the property has been shown for every smaller cardinal. Without loss of generality, suppose that $|I|=\lambda$. Using Lemma~1.6 of \cite{rosicky1994locally} again, we write $I$ as $\bigcup_{\beta<\lambda}I_\beta$. By the proof of Lemma~\ref{facto_size}, there are $J'\subseteq J$ and $K'\subseteq K$, $J'=\bigcup_{\beta<\lambda}J_\beta$, $K'=\bigcup_{\beta<\lambda}K_\beta$ such that $g$ factors through $\iota_{J'}$, $g=\iota_{J'}\circ g'$, and $f$ factors through $\iota_{K'}$, $f=\iota_{K'}\circ f'$, and a chain
\begin{center}
    \begin{tikzcd}
    \Span(f_0,g_0)\arrow[r,""]&\Span(f_1,g_1)\arrow[r,""]&\dots
    \end{tikzcd}
\end{center}
where the horizontal arrows are the canonical ones, which is continuous in the sense that at every limit ordinal, the corresponding $\Span$ is the colimit of the spans before it, and whose colimit is $\Span(f',g')$. The key fact is that the $J_\beta$ and $K_\beta$ are all of size strictly smaller than $\lambda$, which will enable us to use the induction hypothesis. Notice also that we can suppose that $K'=K$, $J'=J$, $f'=f$ and $g'=g$, because we can use Lemma~1.6 of \cite{rosicky1994locally} to get $J=\bigcup_{\beta<\lambda}J'_\beta$ and $K=\bigcup_{\beta<\lambda}K'_\beta$, and replace the $J_\beta$ (resp. $K_\beta$) by $(J_\beta\cup J'_\beta)^*$ (resp. $(K_\beta\cup K'_\beta)^*$), changing the $f_\beta$ and $g_\beta$ accordingly. We define $\lambda_\beta:=max\{|\beta|,|I_\beta|,|J_\beta|,|K_\beta|\}$. Now we inductively build a chain in $\Ind(\mathcal{C})$:
$$D_0\xrightarrow{}D_1\xrightarrow{}\dots$$

such that $D_\beta$ amalgamates $\Span(f_\beta,g_\beta)$ (in a coherent way, i.e. we get a chain of commuting diagrams), and such that $D_\beta$ can be written as a formal colimit of a diagram of size at most $\lambda_\beta$ for all $\beta$. To build $D_0$, just use the induction hypothesis on $\Span(f_0,g_0)$. Notice that as a consequence of Lemmas \ref{facto_size} and \ref{eq_size}, we can choose $D_0$ to be a formal colimit of a diagram of cardinality at most $\lambda_0$. We call this a small amalgamation. At a successor ordinal $\beta=\beta'+1$, first use the induction hypothesis on $\Span(f_\beta,g_\beta)$ to get a small amalgamation $D'$. Now use small amalgamation again for the following diagram:
    % https://q.uiver.app/#q=WzAsNSxbMCwyLCJcXGZjb2xpbV97aVxcaW4gSV97XFxiZXRhJ319QV9pIl0sWzEsMiwiXFxmY29saW1fe2lcXGluIElfXFxiZXRhfUFfaSJdLFswLDEsIkRfe1xcYmV0YSd9Il0sWzEsMSwiRCciXSxbMSwwLCJEX1xcYmV0YSJdLFswLDFdLFswLDJdLFsxLDNdLFszLDQsIiIsMCx7InN0eWxlIjp7ImJvZHkiOnsibmFtZSI6ImRhc2hlZCJ9fX1dLFsyLDQsIiIsMix7InN0eWxlIjp7ImJvZHkiOnsibmFtZSI6ImRhc2hlZCJ9fX1dXQ==
\[\begin{tikzcd}
	& {D_\beta} \\
	{D_{\beta'}} & {D'} \\
	{\fcolim_{i\in I_{\beta'}}A_i} & {\fcolim_{i\in I_\beta}A_i}
	\arrow[dashed, from=2-1, to=1-2]
	\arrow[dashed, from=2-2, to=1-2]
	\arrow[from=3-1, to=2-1]
	\arrow[from=3-1, to=3-2]
	\arrow[from=3-2, to=2-2]
\end{tikzcd}\]
At a limit ordinal $\beta$, first take the colimit $D'$ of the constructed chain, then use Lemmas \ref{facto_size} and \ref{eq_size} to see that there is $S'\subseteq S$ of size $\leq\lambda_\beta$ such that all the $D_{\beta'}\xrightarrow{}D'$ factor through $\fcolim_{s\in S'}M_s=:D_\beta$ in a coherent way (i.e. in such a way that the diagram forms a cocone), so there is a morphism $D'\xrightarrow{}D_\beta$ which we use to extend the chain. Finally, the colimit of the chain $(D_\beta)_{\beta<\lambda}$ is an amalgamation of $\Span(f,g)$, which concludes the induction and the proof.
\end{proof}

More generally, one could define the following:
\begin{definition}
    Let $\mathcal{C}$ be a category. For $d\in\Ind(\mathcal{C})$, write $|d|$ (cardinality of $d$) for the smallest cardinal $\kappa$ such that there is a directed diagram in $\mathcal{C}$ of size $\kappa$ whose colimit in $\Ind(\mathcal{C})$ is $d$.
\end{definition}

Note that, by the proof of 1.5 of \cite{rosicky1994locally}, it would not make a big difference to consider the size of the smallest filtered diagram instead, since an infinite filtered colimit can be written as a directed colimit of the same size. However, some finite filtered colimits cannot be written as less than countable directed colimits, so the definition would not be completely equivalent. The definition with directed colimits seems to be the easiest to manipulate. This definition leads to several interesting questions which are still under investigation, notably regarding the cardinality of colimits in $\Ind(\mathcal{C})$.

\section{Model-theoretic characterizations via amalgamation}\label{sec:model_amalg}
In this section, we prove another condition for a theory to be De Morgan, using model-theoretic tools. We will have to give some model-theoretic definitions, but the reader who is used to first-order model theory should be careful to the differences. We know that the existential formulas are characterized by the property of being preserved by extension. However, in the context of geometric or coherent theories, we are dealing with geometric formulas, which are preserved by any homomorphism: if $M\models \phi(a)$ and $h:M\xrightarrow{} N$, then $N\models \phi(h(a))$. So basically we will replace ``embedding" by ``homomorphism" in all the definitions. Note that this kind of model theory is also studied outside the realm of topos theory, under the name of positive logic. See for example the paper \cite{yaacov2007fondements}, which turned out to be helpful. The models in this section are set-based. We allow ourselves to write $M\models \neg\phi$ for $M\not\models\phi$. $[[x^A.\phi]]_M$ means the interpretation of $\{x^A.\phi\}$ in $M$, i.e. in the set-based case the subset of $M(A)$ of elements satisfying $\phi$. Also remember that a sequent $\phi\vdash_x\psi$ can equivalently be represented by the first-order formula $\forall x.(\phi\Rightarrow\psi)$, and a geometric formula $\{x.\phi\}$ can be represented by the sequent $\top\vdash_x\phi$. 

\subsection{Existentially closed models}
First, we give a proof of a characterization of existentially closed models that appears unproven in \cite{harun1996applications}, and that will be useful later, but is interesting on its own. Let us start with some definitions.
\begin{definition}
    Let $\T$ be a geometric theory, $M\models \mathbb{T}$, we say that $M$ is existentially closed (in $\mathbb{T}$) if for every homomorphism $h:M\xrightarrow{}N$ to a model $N$ of $\T$, for every coherent formula $\phi(a)$ with parameters in $M$, $M\models\phi(a)\Leftrightarrow N\models\phi(h(a))$.
\end{definition}
Notice that since a geometric formula is equivalent to a disjunction of coherent formulas (easy induction) then the above definition can be equivalently formulated by replacing ``coherent formula" by ``geometric formula". However, we will work a lot on coherent theories, so it will be more convenient to express the definition in terms of coherent formulas.
\begin{definition}
    An atomic formula in a signature $\Sigma$ is a formula of the form $R(t_1,\dots,t_n)$ where $R$ is a relation symbol and the $t_i$ are terms.
\end{definition}
\begin{definition}
    For a signature $\Sigma$ and a $\Sigma$-structure $M$, the positive diagram of $M$, noted $Diag^+(M)$, is the set of atomic formulas satisfied by $M$ in the signature $\Sigma_M:=\Sigma\cup\{c_a\,|\,a\in M\}$ where we add a constant symbol $c_a$ for every element $a\in M$.  In other words, for an atomic formula $\phi$ and an element (in fact a tuple) $a\in M$, $\phi(c_a)\in Diag^+(M)$ if and only if $M\models \phi(a)$
\end{definition}

It is a standard result that for $\Sigma$-structures $M$ and $N$, there is a homomorphism $M\xrightarrow{}N$ if and only if $N\models Diag^+(M)$: in this case, we can define the homomorphism $a\in M\mapsto c_a^N$. Note that if $M$ is existentially closed, every $N$ with $h:M\xrightarrow{}N$ is in fact a superstructure of $M$: just apply the definition with $\phi(x,y):=(x=y)$. So being a superstructure of $M$ can be characterized as being a model of the positive diagram $Diag^+(M)$ of $M$.\\
\\
Now we can prove the infamous characterization:

\begin{proposition}\label{exist_closed}
    Let $\mathbb{T}$ be a coherent theory, $M$ a model. Then $M$ is existentially closed if and only if for every coherent formulas $\phi$, $\psi$ such that $\phi\vdash_x \psi$ is provable in $\mathbb{T}$, we have 
    $$[[x.\psi]]_M=[[x.\phi]]_M\cup\bigcup_{\phi'\in\Phi}[[x.\phi']]_M$$
    with $\Phi=\{\phi'\text{ coherent}\,|\,\phi'\vdash_x \psi,\,\phi\land\phi'=0\text{ in $\T$}\}$.
\end{proposition}
\begin{proof}
    Suppose $M$ existentially closed, take $\phi$, $\psi$ as in the theorem, and suppose that $M\models \psi(a)\land\neg\phi(a)$. Then for any homomorphism $h:M\xrightarrow{} N$, we have $N\models \psi(h(a))\land\neg\phi(h(a))$ ($\psi$ is preserved because it is coherent, and $\neg\phi$ because $M$ existentially closed and $\phi$ coherent). So $Diag^+(M)\cup\T\vdash\psi(c_a)\land\neg\phi(c_a)$, so by compactness there is $\phi'(c_a,c_b)\in Diag^+(M)$ such that $T\cup\{\phi'(c_a,c_b)\}\vdash\psi(c_a)\land\neg\phi(c_a)$. As $c_b$ is a constant that does not appear in $\T$, we have $\T\cup\{\exists y\phi'(c_a,y)\}\vdash\psi(c_a)\land\neg\phi(c_a)$, so $\exists y\phi'(x,y)\in\Phi$ and $M\models \exists y\phi'(a,y)$, which concludes.\\
    \\
    In the other direction, consider $M$ with the above property, $\phi$ coherent, $a\in M$, and suppose $h:M\xrightarrow{} N$ with $N\models\phi(a)$ but $M\models \neg\phi(a)$. We apply the property to $\phi\vdash_x\top$, we get that there is some $\phi'\in\Phi$ such that $M\models \phi'(a)$. But $\phi'$ coherent so $N\models \phi'(a)$, contradiction with $\phi\land\phi'=0$.
\end{proof}

Note that we have to assume $\T$ to be coherent in the above theorem, because we use the completeness Theorem~in the proof. We will give a generalization to geometric theories in section \ref{tti}.

\subsection{A Theorem~of Joyal and Reyes}
Another interesting Theorem~is unproven in \cite{harun1996applications}, relating being De Morgan for a theory and the property of having a model companion and satisfying amalgamation. This Theorem~is attributed to Joyal and Reyes, but the proof cannot be found in the litterature. We give a model-theoretic proof, largely inspired from \cite{weispfenning1981model} which only deals with coherent theories whose axioms are of the form $\top\vdash_x\phi$. Then we discuss its topos-theoretic content.

\subsubsection{Model-theoretic proof}
First, as usual, we need some definitions. Here we adapt the usual model-theoretic definitions to the geometric case. 

\begin{definition}
Let $\T$ be a geometric theory: 
\begin{itemize}
    \item $\T^*$ is model consistent relative to $\T$ if every model of $\T$ maps homomorphically to a model of $\T^*$.
    \item $\T^*$ is model complete relative to $\T$ if for every $M_1,\,M_2\models\T^*$, $M\models\T$, $h_i:M\xrightarrow{}M_i$, and for every coherent formula $\phi(a)$ with parameters in $M$, we have $M_1\models\phi(h_1(a))\Leftrightarrow M_2\models \phi(h_2(a))$.
    \item $\T^*$ is model complete if it is model complete relative to itself.
\end{itemize}
\end{definition}

Still with $\T$ geometric we have:

\begin{definition}\label{standard}
\begin{itemize}
    \item $\T^*$ is a model companion of $\T$ if $\T$ and $\T^*$ are mutually model consistent and $\T^*$ is model complete.
    \item $\T^*$ is a model completion of $\T$ if $\T\subseteq\T^*$, $\T^*$ model consistent relative to $\T$ and model complete relative to $\T$.
\end{itemize}
\end{definition}

Clearly a model completion is a model companion. The following classical Lemma~is interesting because it makes the link between the different notions defined above and the amalgamation property. We do the proof for the reader's convenience:

\begin{Lemma}\label{completion_amalg}
    Let $\T$ be a coherent theory, $\T^*$ a model companion of $\T$. Then $\T^*$ is a model completion if and only if the category of models of $\T$ has the amalgamation property.
\end{Lemma}
\begin{proof}
    Suppose $\T^*$ is a model completion. Take $M,\, M_1,\, M_2\models \T$ and a diagram
        \begin{center}
        \begin{tikzcd}
        M_1 &   & M_2\\
            & M \arrow[ul,"h_1"]\arrow[ur, "h_2"']&
        \end{tikzcd}
    \end{center}
    By model consistency there are $M'_i\models \T^*$, $M_i\xrightarrow{h_i'} M'_i$. Since $\T^*$ is model complete relative to $\T$, $M'_1$ and $M'_2$ satisfy the same coherent formulas with parameters in $M$. We just have to show that the theory $Diag^+(M'_1)\cup Diag^+(M'_2)\cup\T$, where constant symbols for elements in $M$ are the same for both diagrams, is consistent, because a model of this theory is exactly an amalgamation of the above diagram. By compactness, we just have to show that for an atomic formula $\phi(c_a)$ with $a\in M'_1$ and $M'_1\models\phi(a)$, the theory $\T\cup\{\phi(c_a)\}\cup Diag^+(M'_2)$ is consistent. But $M'_1\models\exists x.\phi(x)$ so also $M'_2\models\exists x.\phi(x)$. Let $b\in M'_2$ such that $M'_2\models\phi(b)$. A model of $\T\cup\{\phi(c_a)\}\cup Diag^+(M'_2)$ is given by $M'_2$ enriched with the interpretation of $c_a$ by $b$.\\
    \\
    Conversely, take $M\models \T$, $M_1,\, M_2\models \T^*$, $M\xrightarrow{h_i} M_i$. We can amalgamate them by $N\models \T$. Now $N$ maps homomorphically into $N'\models\T^*$, and the homomorphisms $M_i\xrightarrow{}N'$ are elementary on coherent formulas by model completeness. This clearly concludes.
\end{proof}

\begin{corollary}
    Let $\T$ be a coherent theory. Then $\T$ has a model completion if and only if $\T$ has a model companion and the category of models of $\T$ has the amalgamation property.
\end{corollary}

Now in preparation of the proof of the theorem, we should note that two coherent theories are mutually model consistent if and only if they have the same existentially closed models (see Lemma~7 of \cite{yaacov2007fondements} for a proof). This shows that there is a maximal theory, that we will note $\T_{ec}$, which is mutually model consistent with $\T$, namely the coherent theory of the existentially closed models of $\T$. Notice that $\T\subseteq\T_{ec}$. Now it is easy to see that if a theory is model complete (relative to $\T$) then any bigger theory is also model complete (relative to $\T$). So this means that $\T$ has a model companion (completion) if and only if $\T_{ec}$ is a model companion (completion). So we will focus on $\T_{ec}$ from now on.\\
\\
Another important ingredient will be a semantic characterization of the pseudo-complement of an element $\{x.\phi\}$ in $\textup{Sub}_{\mathcal{C}_\T^{coh}}(\{x.\top\})$. Recall that in a presheaf topos $\mathbf{Psh}(\mathcal{C})$, the pseudo-complement of a subobject $A\xrightarrowtail{\enspace}E$ in the Heyting algebra $\textup{Sub}_{\mathbf{Psh}(\mathcal{C})}(E)$ is given by the formula, for all object $c$ of $\mathcal{C}$ and element $x\in E(c)$ 
$$x\in\neg A(c)\Leftrightarrow \forall f:d\xrightarrow{}c,\;E(f)(x)\not\in A(d)$$
\begin{Lemma}\label{pc}
    Let $\T$ be a coherent theory, $\phi$ and $\psi$ two coherent formulas in context $x$. Then $\{x.\psi\}$ is the pseudo-complement of $\{x.\phi\}$ in $\textup{Sub}_{\mathcal{C}_\T^{coh}}(\{x.\top\})$ if and only if for all $M\models \T$ and parameters $a$ in $M$, we have $M\models\psi(a)$ if and only if for any $h:M\xrightarrow{} N$ in $\T\text{-}\Mod(\Set)$, $N\not\models\phi(h(a))$.
\end{Lemma}
\begin{proof}
    Remark~that $\{x.\phi\}$ has always a \emph{geometric} pseudo-complement, that is a pseudo-complement in $\textup{Sub}_{\mathcal{C}_\T}(\{x.\top\})$, which is given by $\bigvee_{\phi'\in\Phi}\phi'$ with $\Phi=\{\phi'|\,\phi\land\phi'=0\}$. And clearly $\phi$ is the pseudo-complement of $\phi$ in $\textup{Sub}_{\mathcal{C}_\T^{coh}}(\{x.\top\})$ if and only if it is logically equivalent to $\bigvee_{\phi'\in\Phi}\phi'$.\\
    \\
    In one direction, if $\T\vdash\psi\Leftrightarrow\bigvee_{\phi'\in\Phi}\phi'$, if $M\models\psi(a)$ then $N\models\phi'(h(a))$ for some $\phi'$ by preservation, so $N\not\models\phi(h(a))$, and conversely if $N\not\models\phi(h(a))$ for all $N,\,h$, then $Diag^+(M)\bigcup \T\not\models\phi(a)$ so by compactness there is a $\phi'\in\Phi$ such that $M\models\phi'(a)$. In the other direction, we just notice that a $\psi$ with this property, if it exists, is unique (up to logical equivalence).
\end{proof}

In other words, the evaluation functor $ev:\mathcal{C}_{\T}^{coh}\xrightarrow[]{}\Set^{\T\text{-}\Mod(\Set)}$, $\{x.\phi\}\mapsto(M\mapsto [[x.\phi]]_M$) preserves pseudo-complements when they exist. Note that a morphism $\theta:\{x.\phi\}\xrightarrow{}\{x.\psi\}$ in $\mathcal{C}_\T^{coh}$ is a formula defining the graph of a function $[[x.\phi]]_M\xrightarrow{}[[x.\psi]]_M$ for every model $M$ of $\T$, which gives a natural action on morphisms for $ev$. Also notice that the previous characterization can be modified by only looking at the models $N$ of $\T_{ec}$, by model consistency and preservation. Now we introduce the following notion from \cite{weispfenning1981model}: 

\begin{definition}
    For $\T$, $\T'$ geometric theories, $\phi$ geometric formula, $\phi$ is strongly invariant in $(\T'\text{-}\Mod(\Set),\,\T\text{-}\Mod(\Set))$ if for every $M\models\T$, $M_1,\,M_2\models\T'$, $h_i:M\xrightarrow{}M_i$, and parameters $a$ in $M$, we have $M_1\models \phi(h_1(a))\Leftrightarrow M_2\models \phi(h_2(a))$. 
\end{definition}
So $\T_{ec}$ is model complete relative to $\T$ if and only if every coherent formula is strongly invariant in $(\T_{ec}\text{-}\Mod(\Set),\,\T\text{-}\Mod(\Set))$. The next Lemma~relates pseudo-complements and strong invariance, and the nice proof is inspired from the proof of Lemma~3.2 of \cite{weispfenning1981model}.

\begin{Lemma}
    Let $\T$ be a coherent theory, $\phi$ a coherent formula. If $\phi$ is strongly invariant in $(\T_{ec}\text{-}\Mod(\Set),\,\T\text{-}\Mod(\Set))$ then it is pseudo-complemented.
\end{Lemma}
\begin{proof}
    We need to add to the langage two unary predicates $M$ and $N$, and a binary predicate $f$. Now consider the theory consisting of the fact that $M$ and $N$ form a partition of the universe, that each $M$ is a model of $\T$ and $N$ a model of $\T_{ec}$ (so in particular the relativisation of $\T$ to $M$ and $\T_{ec}$ to $N$), and that $f$ is a morphism between them. We call this theory $\T^+$. We note $\phi^M$ the relativisation of $\phi$ to $M$. The strong invariance and the fact that $\bigvee_{\phi'\in\Phi}\phi'$ ($\Phi$ as above) satisfies the property of the pseudo-complement of $\phi$ gives that 
    $$\T^+\vdash \forall x(M(x)\xrightarrow{}(\neg\phi^N(f(x))\leftrightarrow\bigvee_{\phi'\in\Phi}\phi'^M(x)))$$
    Now just use compactness to extract a finite subset of $\Phi$ that satisfies this, its disjunction is clearly a pseudo-complement of $\phi$.
\end{proof}

Now the Theorem~of Joyal and Reyes will be a direct consequence of the following, which corresponds to Theorem~3.3 of \cite{weispfenning1981model}. We write $\phi^*$ for the coherent pseudo-complement of $\phi$ if it exists.

\begin{theorem}\label{prejoyal}
    Under the same assumptions, $\phi^*$ and $\phi^{**}$ exist and $\phi^*\lor\phi^{**}=1$ if and only if $\phi$ is strongly invariant in $(\T_{ec}\text{-}\Mod(\Set),\,\T\text{-}\Mod(\Set))$.
\end{theorem}
\begin{proof}
    Suppose $\phi$ strongly invariant in $(\T_{ec}\text{-}\Mod(\Set),\,\T\text{-}\Mod(\Set))$. By the previous Lemma~$\phi^*$ exists. We have to show that it is strongly invariant. Consider $M\models\T$, $a$ parameters in $M$, $M_1,\,M_2\models\T_{ec}$, $h_i:M\xrightarrow{}M_i$. Suppose $M_1\models\phi^*(h_1(a))$, then $M_1\not\models\phi(h_1(a))$, and by strong invariance  $N\not\models\phi(h(h_2(a)))$ for all $h:M_2\xrightarrow{}N\models\T_{ec}$, and so $M_2\models\phi^*(h_2(a))$. So indeed $\phi^*$ strongly invariant and $\phi^{**}$ exists. Now we still need to prove that $\phi^*\lor\phi^{**}=1$. So suppose $M\not\models\phi^*(a)$. This means that there is $h:M\xrightarrow{}N\models\T_{ec}$, $N\models\phi(h(a))$. By strong invariance it is true for all $h,\,N$. So for all $h,\, N$, $N\not\models\phi^*(h(a))$, so $M\models\phi^{**}(a)$.\\
    \\
    Conversely, take $M\models\T$, $a$ parameters in $M$, $M_1,\,M_2\models\T_{ec}$, $h_i:M\xrightarrow{}M_i$, and suppose $M_1\models\phi(h_1(a))$. So $M\not\models\phi^*(a)$, so $M\models\phi^{**}(a)$, so $M_2\not\models\phi^*(h_2(a))$. So there is $h:M_2\xrightarrow{}N$ in $\T_{ec}\text{-}\Mod(\Set)$ such that $N\models\phi(h(h_2(a)))$. To finish the proof we need to show that $\phi$ is invariant in $\T_{ec}\text{-}\Mod(\Set)$. But by Proposition~\ref{exist_closed}, $\T_{ec}\vdash \phi\vee\phi^*$, so this is clear.
\end{proof}

Notice that we did use the explicit definition of $\T_{ec}$ in the last proof. In fact, if all the pseudo-complements of coherent formulas exist, Proposition~\ref{exist_closed} says that the existentially closed models of $\T$ are exactly axiomatised by the `coherent Booleanization' of $\T$, i.e. by adding all the axioms $\phi\vee\phi^*$ for every coherent formula $\phi$, and this is always a model companion of $\T$, because all its models are existentially closed. We will do the link with the Booleanization of \cite{caramello2009morgan} in the next section. Now we see that we proved the desired theorem.

\begin{theorem}\label{Joyal}
    For a coherent theory $\T$, $\T$ satisfies De Morgan's law for coherent formulas if and only if $\T$ has a model completion.
\end{theorem}

Here ``$\T$ satisfies De Morgan's law for coherent formulas" means that $\T$ proves $\phi^*\lor\phi^{**}$ for every coherent $\phi$. Equivalently, by Lemma~\ref{completion_amalg} and the fact that a model completion is a model companion, $\T$ is De Morgan for coherent formulas if and only if $\T$ has a model companion and satisfies amalgamation. Note that this is not equivalent to the theory being De Morgan, because this last property requires De Morgan's law for all geometric formulas by Theorem~\ref{dmltheory}. However, we will give the right generalization in the next section. 

\subsubsection{Topos-theoretic interpretation}\label{tti}
Since, as mentioned before, for a coherent theory, the property of satisfying De Morgan's law for coherent formulas only does not imply being De Morgan (i.e. having a De Morgan classifying topos), we do not expect to fully interpret the Theorem~of the previous section topos-theoretically. However, under reasonable assumptions, we can generalize some of its aspects to geometric theories. First, we elucidate the link between Booleanization and model companions. In this section, for a geometric formula $\phi$, we will note by $\phi^*$ the \emph{geometric} pseudo-complement of $\phi$, which, as we saw, always exists.

\begin{definition}
    Let $\T$ be a geometric theory. We define $\textup{Bool}(\T)$ to be the theory $\T\cup\{\top\vdash_x\phi\lor\phi^*\,|\,\phi\text{ geometric}\}$.
\end{definition}

Equivalently, $\textup{Bool}(\T)$ can be obtained from $\T$ by adding the sequents $\top\vdash_x\phi$ for every geometric formula $\phi$ such that for every geometric $\psi$ such that $\psi\vdash_x\bot$ is not provable in $\T$, $\psi\land\phi\vdash_x\bot$ is not provable in $\T$. Indeed, for any geometric $\phi$ and $\psi\ne\bot$, if $\phi\land\psi\leq\bot$, then $\psi\leq\phi^*$, so $\psi\land\phi^*=\psi\ne\bot$, so $\psi\land(\phi\lor\phi^*)=\psi\ne\bot$, and conversely if $\phi\land\psi\ne\bot$ for all $\psi$ then by the explicit description of $\phi^*:=\bigvee_{\phi'\land\phi=\bot}\phi'$, we have $\phi^*=\bot$ so $\phi=\phi\lor\phi^*$. Now remembre that we have a `canonical' Boolean topos associated to any topos.

\begin{proposition}[\cite{caramello2009morgan} Theorem~5.7]
    With the same assumptions, if $\T$ is classified by $\mathcal{E}$, then $\textup{Bool}(\T)$ is classified by $\mathbf{sh}_{\neg\neg}(\mathcal{E})$.
\end{proposition}

We can now do the link with existentially closed models.

\begin{proposition}
    Let $\T$ be a geometric theory such that the evaluation functor $ev:\mathcal{C}_{\T}\xrightarrow[]{}Set^{\T\text{-}\Mod(\Set)}$ preserves pseudo-complements. Then the Booleanization $\textup{Bool}(\T)$ of $\T$ is a model companion of $\T$.
\end{proposition}
\begin{proof}
    First, we can see that a model of $\textup{Bool}(\T)$ is exactly a model of $\T$ which is existentially closed. Indeed, one direction is clear, for the other suppose $M\models\T$ existentially closed such that $M\not\models\phi(a)$, we want to show that $M\models\phi^*(a)$, equivalently that for all $h:M\xrightarrow{}N$, $N\not\models\phi(h(a))$ because $ev$ preserves pseudo-complements, but this is true by existential closeness.\\
    \\
    Now we just need to show that every model of $\T$ is mapped homomorphically to an existentially closed one. This is achieved by using that $\T\text{-}\Mod(\Set)$ admits directed colimits (see for example \cite{borceux1994handbook} Corollary~4.3.2). Indeed, take $M\models\T$, and an ordering of the set of coherent formulas with parameters in $M$, noted $(\phi_\alpha(a_\alpha))_\alpha$. If there is $h:M\xrightarrow{}N$ such that $N\models\phi_0(a_0)$, define $N_1:=N$, $h_1:=h$, else take the identity on $M$. Continue inductively, taking (directed) colimits at limit stages. Let $M_1$ be the final such model, by construction for any $\phi_\alpha(a_\alpha)$, if $h:M_1\xrightarrow{}N$ with $N\models\phi_\alpha(a_\alpha)$, then $N_{\alpha+1}\models\phi_\alpha(a_\alpha)$, so $M_1$ also does. To conclude we just need to do this countably many times and take the colimit.
\end{proof}

\begin{remark}
    In fact this is a particular case of \cite{kamsma2024existentially}, Theorem~7.1. Indeed, $ev$ preserves pseudo-complements implies as we noticed that a model is existentially closed if and only if it is a model of the Booleanization of $\T$, which in turn is equivalent to being $(\mathcal{X,B})$-s.e.c in the sens of \cite{kamsma2024existentially}, for $\mathcal{X}$ the representables and $\mathcal{B}$ the representable subobjects. Now (ii) of 7.1 \cite{kamsma2024existentially} is exactly the definition of model companion: $\Set$-entwined means that the corresponding theories are mutually model-consistent, and Proposition~6.11 of \cite{kamsma2024existentially} shows that if a topos is locally zero-dimensional then all the models of the theory it classifies are existentially closed. This thus gives our result. Interestingly, this implies that the Booleanization has enough points, so it is atomic (see C3.5.2 \cite{johnstone2002sketches}). 
\end{remark}

Now, we can give a generalization of one direction of Theorem~\ref{Joyal} using this proposition.

\begin{theorem}
    Under the same assumptions as the last proposition, if $\T$ is additionally De Morgan, then $\textup{Bool}(\T)$ is a model completion of $\T$.
\end{theorem}
\begin{proof}
    Basically the same as the corresponding direction of Theorem~\ref{prejoyal}.
\end{proof}

Now we look at the other direction of Theorem~\ref{Joyal}. We still need to assume that $ev$ preserves pseudo-complements. Since we rely on the completeness theorem, we should assume that (the classifying topos of) $\T$ has enough points.

\begin{theorem}\label{joyal_geom}
    Let $\T$ as before. If $\T$ has enough points and $\T$-$\Mod(\Set)$ has the amalgamation property, then $\T$ is De Morgan.
\end{theorem}
\begin{proof}
    Take $M\models\T$, $M\not\models\phi^*(a)$, i.e. there is $h:M\xrightarrow{}N$ such that $N\models\phi(h(a))$. We want to show that $M\models\phi^{**}(a)$, i.e. for all $h':M\xrightarrow{}N'$, there is $f:N'\xrightarrow{}N''$, $N''\models\phi^*(f(h'(a)))$. Amalgamating $h$ and $h'$ gives the result.
\end{proof}

\subsubsection{Link with theories of presheaf type}

 In some sense, the two previous theorems are not very surprising, as we just identified what made the theorem work for coherent theories and assumed it as hypothesis. Now we sketch a more conceptual trail. We rely on~\cite{espindola2023every}.
 
 Recall that a theory of presheaf type is De Morgan if and only if its category of finitely presented models satisfies amalgamation. Clearly this is very similar to Theorem~\ref{Joyal} (in fact its reformulation just after it). So if we could find a way to `make any theory into a theory of presheaf type', maybe we would have a proof of Theorem~\ref{Joyal} or of some interesting generalization. In fact,~\cite{espindola2023every} shows that (under set theoretic assumptions) every theory is `eventually' of presheaf type. This uses \emph{$\lambda$-geometric logic}, which is the extension of geometric logic within infinitary first order logic, where we permit contexts and conjunctions of infinite ($<\lambda$) length in formulas, and where we assume the \emph{transfinite transitivity rule}: 
 \begin{center}
$\infer{\phi_0\vdash_{y_0} \bigvee_{f\in\gamma^\gamma}\exists_{\beta<\gamma} {x_f}_{|\beta+1}\bigwedge_{\beta<\gamma}{\phi_f}_{|\beta} }
    {\arraycolsep=1.4pt\begin{array}{l}
        \phi_f\vdash_{y_f} \bigvee_{g\in\gamma^{\beta+1},\,g_{|\beta}=f}\exists x_g\phi_g\quad \beta<\gamma,\,f\in\gamma^\beta \\
        \phi_f \dashv\vdash_{y_f} \bigwedge_{\alpha<\beta}{\phi_f}_{|\alpha} \quad \beta<\gamma,\,\text{limit}\,\beta,\,f\in \gamma^\beta
    \end{array}}$
   \end{center} 
   (see~\cite{espindola2020infinitary} for a complete description of this logic). A \emph{$\lambda$-geometric category} is a $\lambda$-complete geometric category such that compositions of $\lambda$-transfinite jointly covering $\lambda$-families are covering~\cite[Definition 2.1.1]{espindola2020infinitary}. A \emph{$\lambda$-topos} is a topos which is a $\lambda$-geometric category. The \emph{$\lambda$-classifying topos} of a $\lambda$-geometric theory is the $\lambda$-topos satisfying the universal property of the classifying topos restricted to $\lambda$-toposes and geometric morphisms whose inverse image preserves $\lambda$-limits~\cite[Theorem 3.15]{espindola2023every}.
 
 It is proven in~\cite[Theorem 5.5]{espindola2023every} that given a geometric theory $\T$ and an inaccessible cardinal (whose existence is not guaranteed by ZFC) $\lambda$ which is big enough, in particular bigger than the signature of $\T$, we have a commutative diagram:

 \begin{center}
     \begin{tikzcd}
         \mathcal{C}_\T \arrow[rr, "ev_\lambda"]\arrow[d, "\phi"'] & & \Set^{\T\text{-}\Mod(\Set)_{<\lambda}}\\
         (\Lex(\mathcal{C}_\T,\Set)_{<\lambda})^{op}\arrow[d,"\yo"'] & & \\
         \Sh((\Lex(\mathcal{C},\Set)_{<\lambda})^{op},J)\arrow[uurr, "\cong"']&&
     \end{tikzcd}
 \end{center}

\noindent We need to explain some notations:
\begin{itemize}
    \item ``$<\lambda$'' means ``pointwise of cardinality smaller than $\lambda$'' for the category $\Lex(\mathcal{C},\Set)_{<\lambda}$ and ``of cardinality smaller than $\lambda$'' for the category of models $\T\text{-}\Mod(\Set)$.
    \item $\Lex(\mathcal{C}_\T,\Set)$ is the category of left exact functors from $\mathcal{C}_\T$ to $\Set$. In fact, by the first point of~\cite[Theorem 4.7]{espindola2023every} we have that $(\Lex(\mathcal{C}_\T,\Set)_{<\lambda})^{op}$ is essentially the syntactic category of $\T$ seen as a $\lambda$-geometric theory. This is reasonable, since it is the free completion of $\mathcal{C}_\T$ with $<\lambda$ limits: remember that products are used to interpret context extension, and pullbacks to interpret conjunctions. We will thus note it $\mathcal{C}_\T^{\lambda\text{-}geom}$. 
    \item $\Sh((\Lex(\mathcal{C}_\T,\Set)_{<\lambda})^{op},J)$ is the $\lambda$-classifying topos of $\T$, with $J$ the expected coverage, which is in particular subcanonical.
\end{itemize}
Now $\yo$ preserves pseudo-complements when they exist~\cite[Lemma 3.1]{butz1998classifying}, and it is full and faithful, so $ev_\lambda$ preserves pseudo complements if and only if $\phi$ (which is also just the Yoneda embedding) preserves them. So we can give a purely topos-theoretic point of view on Theorem~\ref{joyal_geom}: if we have amalgamation in $(\T\text{-}\Mod(\Set))_{<\lambda}$ for some big enough inaccessible cardinal $\lambda$, the topos $\Sh((\Lex(\mathcal{C}_\T,\Set)_{<\lambda})^{op},J)$ is De Morgan, so all the $\Sub_{(\Lex(\mathcal{C}_\T,\Set)_{<\lambda})^{op}}(\{\mathbf{x}.\top\})$ (where $x$ is possibly infinite) are De Morgan algebras, by the same arguments we used for geometric logic. In particular, if $\phi$ (equivalently $ev_\lambda$) preserves pseudo-complements, since it is also conservative and preserves disjunctions, we directly get that $\T$ is De Morgan, i.e.~the $\Sub_{\mathcal{C}_\T}(\{\mathbf{x}.\top\})$ for a finite context $\mathbf{x}$ are De Morgan. Also now we have a sufficient condition for having amalgamation in $\T\text{-}\Mod(\Set)_{<\lambda}$. Summarizing, we have:
\begin{theorem}
Let $\T$ be a geometric theory in a signature $\Sigma$, $\lambda$ inaccessible cardinal, $\lambda>2^{|\Sigma|}$. Then the category of set-based models of cardinality $<\lambda$ of $\T$ verifies amalgamation if and only of the $\Sub_{\mathcal{C}_\T^{\lambda\text{-}geom}}(\{\mathbf{x}.\top\})$ (with $|x|<\lambda$) are De Morgan algebras.
\end{theorem}

\begin{corollary}
    Let $\T$ be a geometric theory in a signature $\Sigma$, $\lambda$ an inaccessible cardinal, $\lambda>2^{|\Sigma|}$, and suppose that $ev_\lambda$ preserves pseudo-complements. If the category of set-based models of cardinality $<\lambda$ of $\T$ verifies amalgamation, then $\T$ is De Morgan.
\end{corollary}

\begin{remark}\label{rem:gu}
    Assuming the Grothendieck Universe axiom (which is reasonable~\cite{maddy1988believing}), we see that condition~\ref{def:spc} becomes purely syntactic. Indeed, recall that $ev_\lambda$ preserves pseudo-complements if and only if the inclusion $\phi:\mathcal{C}_\T\to \mathcal{C}_\T^{\lambda\text{-}geom}$ preserves pseudo-complements, which is a syntactic property. So we just need to prove that $ev$ preserves pseudo-complements if and only if $ev_\lambda$ preserves pseudo-complements for some big enough inaccessible $\lambda$. One direction is easy. For the other one, first notice that the geometric formulas over some signature $\Sigma$ form a set by~\cite[Lemma D1.3.8]{johnstone2002sketches}, so for every context $\mathbf{x}$ and every pair $\{\mathbf{x}.\phi\}$, $\{\mathbf{x}.\psi\}$ which are not pseudo-complements, we can choose either a model $M$ with parameters $a$ such that $M\not\models\psi(a)$ and for any $h:M\xrightarrow{} N$ in $\T\text{-}\Mod(\Set)$, $N\not\models\phi(h(a))$, or a pair $(M,a)$ and a homomorphism $h:M\to N$ with $N$ another model of $\T$ such that $M\models\psi(a)$ and $N\not\models\phi(h(a))$. The collection of all the chosen models $M$ and $N$ form a set, so we can just take $\lambda$ to be an inaccessible cardinal bigger that the sup of all their sizes. 
\end{remark}

\noindent On the other hand,~\cite[Theorem 5.5]{espindola2023every} suggests that going to $\lambda$-geometric logic makes the `finitely presented' models of $\T$ exactly the models of size $<\lambda$, because we can arbitrarily take infinite ($<\lambda$) conjunctions of formulas. More formally, we can prove: 
\begin{proposition}\label{prop:lambdapres}
    Let $\Sigma$ a signature, $\lambda$ a regular cardinal bigger than $|\Sigma|$, $\T$ a geometric theory, $M$ a model of $\T$. Then $M$ is presented by a $\lambda$-geometric formula $\phi(a)$ if and only if $M$ is of size $<\lambda$.
\end{proposition}
\begin{proof}
    One direction is easy: if $M$ is of size $<\lambda$, we can just take $a$ to be a tuple enumerating all the elements of $M$, and $\phi(a)$ to be the conjunction of all the atomic formulas satisfied by $a$. Conversely, let $M=M_{\phi(a)}$. By the definition of a $\lambda$-geometric formula, $a$ is a tuple of size $<\lambda$. In addition, it is easy to show, as in the case of geometric logic, that $\phi$ is equivalent to an infinite disjunction of formulas of the form $\exists\mathbf{x}_\gamma \bigAND_{i\in I}\phi_i$, where $\gamma,\,|I|<\lambda$ and $\phi_i$ are atomic formulas. Since $M\models \phi(a)$, then $a$ makes one of these formulas true, i.e.~there is $b\in M$ a tuple of size $<\lambda$ such that $M\models (\bigAND_{i\in I}\phi_i)(a,b)$. Let $<a,b>$ be the substructure generated by $a$ and $b$. Since $\lambda$ is bigger than the cardinality of the language, $<a,b>$ is of size $<\lambda$, so by (a similar reasoning as in the proof of) downward Löwenheim-Skolem, we can find $N$ a submodel of $M$ of same size $<\lambda$ containing $<a,b>$ such that the inclusion $N\subseteq M$ reflects coherent formulas, equivalently geometric formulas, equivalently geometric sequents. Now since $M$ is presented by $\phi(a)$, there is a morphism $f:M\xrightarrow{}N$ sending $a$ to $a$. Moreover, if we compose this morphism with the inclusion $N\subseteq M$, we get a morphism $M\xrightarrow{}M$ sending $a$ to $a$. But there is a unique such morphism which is the identity morphism. So $f$ is injective, which automatically gives $|M|<\lambda$ (in fact $M\cong N$).
\end{proof}

\noindent We also have in the setting of~\cite{espindola2023every} that the $<\lambda$ models are jointly conservative by the choice of $\lambda$, so the main Theorem 5.5 therein, summarized by the diagram above, seems to be a generalization of~\cite[Corollary 9.1.3]{caramello2018theories} to $\lambda$-geometric logic, where we just replace ``finite'' by ``$<\lambda$''. We state this theorem here for the readers convenience. 
\begin{theorem}[\cite{caramello2018theories} Corollary 9.1.3]
    Let $\Sigma$ be a finite signature and $K$ a category of finite $\Sigma$-structures
and $\Sigma$-structure homomorphisms between them. Then the common geometric theory
of $K$, $Th(K)$, is of presheaf type, classified by the topos $\Psh(K^{op})$, and its
finitely presentable models are precisely the finite ones.
\end{theorem}

\noindent This point of view seems conceptually interesting, and could lead to new results in the same vein of~\cite{espindola2023every}. 

\subsection{Amalgamation via finitely presented structures}

As already mentioned, a theory of presheaf type is De Morgan if and only if $\textup{f.p.}\mathbb{T}\text{-}\textup{mod}(\Set)$ has the amalgamation property. In \cite{bagchi1992morgan}, Bagchi gives an analogue of this result for coherent theories, using purely model-theoretic tools. Let us start with some definitions. In this section, $\mathbb{T}$ is a relational coherent theory. We denote by $FP(\mathbb{T})$ the finitely presented structures that embed in a model of $\mathbb{T}$.
\begin{definition}
    Let $G\in FP(\mathbb{T})$, $g\in G^n$, and a diagram in $FP(\mathbb{T})$ 
    \begin{center}
            \begin{tikzcd}
        G \arrow[r,""]\arrow[d,""] & G_1 \\
        G_2
    \end{tikzcd}
    \end{center}
    An amalgamation in $Mod_\mathbb{T}(\Set)$ is a model $A$ of $\mathbb{T}$ and morphisms $G_i\xrightarrow{} A$ which make the following diagram commute.
    \begin{center}
            \begin{tikzcd}
        n\arrow[r,"g"] & G \arrow[r,""]\arrow[d,""] & G_1 \arrow[d,""] \\
        & G_2\arrow[r,""] & A
    \end{tikzcd}
    \end{center}
\end{definition}
\begin{definition}
    Let $G,G_1,G_2\in FP(\mathbb{T})$, $g\in G^n$. We say that $(G_1,G_2)$ is a base of conditions over $(G,g)$ if there is a diagram as above that cannot be amalgamated.
\end{definition}
\begin{definition}
    Let $A\in Mod_\mathbb{T}()$, $a\in A^n$. We say that $(A,a)$ allows amalgamation over conditions in $Mod_\mathbb{T}(\Set)$ if there are $(G,g)$ and $h:G\xrightarrow{}A$ such that $a=h(g)$ and there is no base of conditions over $(G,g)$. We say that $A$ allows amalgamation over conditions if $(A,a)$ allows amalgamation over conditions for all $n$, $a\in A^n$.
\end{definition}
We can now state Bagchi's result.
\begin{theorem}[\cite{bagchi1992morgan}]\label{bagchi2}
$\mathbb{T}$ is De Morgan if and only if every $A\in Mod_\mathbb{T}(\Set)$ allows amalgamation over conditions.    
\end{theorem}

The proof is purely model-theoretic. We can state a more general version of this for $I_r$, by considering amalgamation of diagrams with $r+1$ legs. We want to interpret this topos-theoretically, and eventually generalize it to geometric theories.

\begin{definition}
    Let $\T$ be a geometric theory which is an extension of a theory $\T'$ of presheaf type in the same language. We note $FP(\T',\T)$ the set of finitely presented models of $\T'$ which admit a morphism into a model of $\T$.
\end{definition}

\begin{definition}\label{relative_amalg}
    Let $\T'\subseteq\T$ geometric theories, $\T'$ of presheaf type. We say that $\T$ has the relative amalgamation property with respect to $\T'$ if for every model $A\models\T$, for every $a:G'\xrightarrow{}A$ with $G'\in FP(\T',\T)$, there is $h:G\xrightarrow{}A$ with $G\in FP(\T',\T)$ and a morphism $g:G'\xrightarrow{}G$ over $A$ such that every diagram as bellow can be amalgamated
    \begin{center}
            \begin{tikzcd}
        G'\arrow[r,"g"] & G \arrow[r,""]\arrow[d] & G_1 \arrow[d, dotted] \\
        & G_2\arrow[r, dotted] & B
    \end{tikzcd}
    \end{center}
    where $G_1,\, G_2\in FP(\T',\T)$, $B\models\T$.
\end{definition}

We would like to characterize the theories $\T$ such that $\T$ is De Morgan if and only if $\T$ has the relative amalgamation property with respect to some theory $\T'$, for example the empty theory in the same language. Bagchi essentially proved that this is true for coherent relational countable theories (see below). In fact, the relative amalgamation property condition can be interpreted in the following way: one would like to characterize DML via amalgamation of finitely presented models, but in the general case there are not enough such models, so we should look at finitely presented models of a smaller theory instead. As a consequence, one would expect that theories of presheaf type indeed satisfy DML if and only if they have the relative amalgamation property. This is true in the case where the finitely presented models of $\T$ are finitely presented as models of $\T'$, which is a reasonable assumption in this case. In fact, this is almost tautological.

\begin{theorem}
    Let $\T'\subseteq\T$ geometric theories of presheaf type. Suppose that the finitely presentable models of $\T$ are finitely presentable as models of $\T'$. Then $\T$ is De Morgan if and only if it has the relative amalgamation property with respect to $\T'$.
\end{theorem}
\begin{proof}
    Suppose that $f.p.\T$-$\Mod(\Set)$ has the amalgamation property. Consider $A\models\T$, $a:G'\xrightarrow{}A$ with $G'\in FP(\T',\T)$. Recall that $A$ can be written as a filtered colimit of finitely presented models of $\T$. Now $A$ is a model of $\T'$ and $G'$ is a finitely presented model of $\T'$, so $a$ factors through some $h:G\xrightarrow{}A$ with $G$ finitely presented model of $\T$, with $a=h\circ g$ for some $g:G'\xrightarrow{}G$. Now consider a diagram  
    \begin{center}
            \begin{tikzcd}
        G'\arrow[r,"g"] & G \arrow[r,""]\arrow[d] & G_1 \\
        & G_2 & 
    \end{tikzcd}
    \end{center}
    with $G_1,\, G_2\in FP(\T',\T)$, i.e. there is $g_1:G_1\xrightarrow{}A_1$ and $g_2:G_2\xrightarrow{}A_2$ with $A_1,\,A_2\models\T$. $g_i$ factors through some finitely presented model $G'_i$ of $\T$, and we get a diagram 
     \begin{center}
            \begin{tikzcd}
        G'\arrow[r,"g"] & G \arrow[r,""]\arrow[d] & G'_1 \\
        & G'_2 & 
    \end{tikzcd}
    \end{center}
    by composition, and this diagram admits amalgamation by hypothesis because $G,\,G'_1,\,G'_2$ are all finitely presented models of $\T$.\\
    \\
    Conversely, consider a diagram 
     \begin{center}
            \begin{tikzcd}
        G' \arrow[r,"g_1"]\arrow[d, "g_2"'] & G_1 \\
        G_2 & 
    \end{tikzcd}
    \end{center}
    in $f.p.\T$-$\Mod(\Set)$. By hypothesis, we can write the identity morphism $G'\xrightarrow{}G'$ as $$G'\xrightarrow{g}G\xrightarrow{h}G'$$
    and we can amalgamate
    \begin{center}
            \begin{tikzcd}
        G'\arrow[r,"g"] & G \arrow[r,"h"] & G' \arrow[r,""]\arrow[d] & G_1 \arrow[d, dotted] \\
       & & G_2\arrow[r, dotted] & B
    \end{tikzcd}
    \end{center}
    This easily implies that we can amalgamate the diagram by a finitely presented model.
\end{proof}
Notice that we do not really need $G'$ in definition \ref{relative_amalg} to be finitely presentable as model of $\T'$, but just `relatively to models of $\T$'. With this assumption, we then do not need the finitely presentable models of $\T$ to be finitely presentable as models of $\T'$.\\
\\
However, this is not completely satisfying because Bagchi's Theorem~\ref{bagchi2} does not exactly say that a coherent relational countable theory is De Morgan if and only if it has the relative amalgamation property with respect to the empty theory. So we give another version, which is equivalent for theories of presheaf type.

\begin{definition}
    Let $\T'\subseteq\T$ geometric theories, $\T'$ of presheaf type. We say that $\T$ has the weak relative amalgamation property with respect to $\T'$ if for every model $A\models\T$, $A$ is the colimit of a filtered diagram $(G_i)_{i\in I}$ of models of $\T'$, such that for every $G_i\xrightarrow{}A$, there is $h:G\xrightarrow{}A$ with $G$ finitely presented model  of $\T'$ and a morphism $g:G_i\xrightarrow{}G$ over $A$ such that every diagram as bellow can be amalgamated
    \begin{center}
            \begin{tikzcd}
        G_i\arrow[r,"g"] & G \arrow[r,""]\arrow[d] & H_1 \arrow[d, dotted] \\
        & H_2\arrow[r, dotted] & B
    \end{tikzcd}
    \end{center}
    where $H_1,\, H_2\in FP(\T',\T)$, $B\models\T$.
\end{definition}

\begin{theorem}
    Let $\T'\subseteq\T$ geometric theories of presheaf type. Suppose that the finitely presentable models of $\T$ are finitely presentable as models of $\T'$. Then $\T$ is De Morgan if and only if it has the weak relative amalgamation property with respect to $\T'$.
\end{theorem}
\begin{proof}
    Suppose that $f.p.\T$-$\Mod(\Set)$ has the amalgamation property. Consider $A\models\T$. Recall that we can write $A=\colim_{i\in I}G_i$ with $G_i$ finitely presented model of $\T$. By assumption, the $G_i$ are also finitely presented models of $\T'$. Now consider a diagram  
    \begin{center}
            \begin{tikzcd}
        G_i\arrow[r,"\id"] & G_i \arrow[r,""]\arrow[d] & H_1 \\
        & H_2 & 
    \end{tikzcd}
    \end{center}
    with $H_1,\, H_2\in FP(\T',\T)$, i.e. there is $h_1:H_1\xrightarrow{}A_1$ and $h_2:H_2\xrightarrow{}A_2$ with $A_1,\,A_2\models\T$. $h_i$ factors through some finitely presented model $H'_i$ of $\T$, and we get a diagram 
     \begin{center}
            \begin{tikzcd}
        G_i\arrow[r,"\id"] & G_i \arrow[r,""]\arrow[d] & H'_1 \\
        & H'_2 & 
    \end{tikzcd}
    \end{center}
    by composition, and this diagram admits amalgamation by hypothesis because $G,\,H'_1,\,H'_2$ are all finitely presented models of $\T$.\\
    \\
    Conversely, consider a diagram 
     \begin{center}
            \begin{tikzcd}
        H \arrow[r,"h_1"]\arrow[d, "h_2"'] & H_1 \\
        H_2 & 
    \end{tikzcd}
    \end{center}
    in $f.p.\T$-$\Mod(\Set)$. Write $H=\colim_{i\in I}G_i$ as given by hypothesis. The identity morphism $H\xrightarrow{}H$ factors through some $G_i$:
    $$\id=H\xrightarrow{f}G_i\xrightarrow{h}H$$
    and there is $g:G_i\xrightarrow{}G$ through which $h$ factors and we can amalgamate
    \begin{center}
            \begin{tikzcd}
        G_i\arrow[r,"g"] & G \arrow[r,""] & H \arrow[r,""]\arrow[d] & H_1 \arrow[d, dotted] \\
       & & H_2\arrow[r, dotted] & B
    \end{tikzcd}
    \end{center}
    We can then precompose with $f$ to get $\id_H$ before the square. This easily implies that we can amalgamate the diagram by a finitely presented model.
\end{proof}

This proves in particular the weak relative amalgamation and relative amalgamation are equivalent in this case. Moreover, we have:
\begin{theorem}[Bagchi]
    Let $\T$ be a coherent countable relational theory. Then $\T$ is De Morgan if and only if it has the weak relative amalgamation property with respect to the empty theory in the same language.
\end{theorem}
\begin{proof}
    We need to prove that the weak relative amalgamation is equivalent to the condition in Theorem~\ref{bagchi2}. This just comes from the fact that a model $A\models\T$ can be written as the directed colimit of its finite substructures. Indeed, since the language is relational, any finite subset of $A$ can be turned into a finite substructure. 
\end{proof}

This kind of relative characterization of De Morgan's law may be interesting in connection with relative topos theory \cite{caramello2021relative}.

\section{Enforcing De Morgan's law}\label{sec:enforce}
\subsection{Gleason's cover}
We mentioned at the end of section \ref{sec:dml} that for any topos $\mathcal{E}$, there exists $\mathcal{E}_m\hookrightarrow{}\mathcal{E}$ such that $\mathcal{E}_m$ is De Morgan. However, there is another way of getting a De Morgan topos starting from $\mathcal{E}$, and we get a surjection instead of an inclusion. This is Gleason's cover, first invented for topological spaces by Gleason, and then generalized to toposes by Johnstone in \cite{johnstone1980gleason}. Today, the development of the theory of relative toposes \cite{caramello2021relative} and of existential sites \cite{caramello2022fibred} gives us better tools to study this construction in a more concrete and explicit way. Essentially, the goal of this section is to get back all the results of Jonstone on Gleason's cover using these techniques (which is not staightforward). We hope that this will eventually lead to new results on the subject.

\subsubsection{Generalities on internal locales}\label{sec:int_loc}
We start by recalling some notions of \cite{caramello2021relative} and \cite{caramello2022fibred}. 
\begin{definition}
    For a category $\mathcal{C}$, a $\mathcal{C}$-indexed category is a pseudo-functor $\mathcal{C}^{op}\xrightarrow{}\Cat$.
\end{definition}
Now recall the \emph{Grothendieck construction} on a $\mathcal{C}$-indexed category $\dcat$, noted $\mathcal{G}(\dcat)$, which is the category whose objects are pairs $(c,x)$ with $c$ an object of $\mathcal{C}$ and $x$ an object of $\dcat(c)$, and whose morphisms $(d,y)\xrightarrow{}(c,x)$ are pairs $(f,\alpha)$ with $f:d\xrightarrow{}c$ in $\mathcal{C}$ and $\alpha:y\xrightarrow{}\dcat(f)(x)$ in $\dcat(d)$. We note $\pi_\dcat$ the obvious projection functor $\mathcal{G}(\dcat)\xrightarrow{}C$. 
\begin{definition}[\cite{caramello2021relative} Definition 8.2.1]
    Let $(\mathcal{C},J)$ be a site. A relative site over $(\mathcal{C},J)$ is a comorphism of sites $\pi_\dcat:(\mathcal{G}(\dcat),K)\xrightarrow{}(\mathcal{C},J)$ for some $\mathcal{C}$-indexed category $\dcat$ and topology $K$. 
\end{definition}
Given a $\mathcal{C}$-indexed category $\dcat$ and a topology $J$, there is a topology, known as the \emph{Giraud topology} (relative to $J$), which is the smallest topology $K$ on $\mathcal{G}(\dcat)$ such that $\pi_\dcat$ is a comorphism of sites. Recall that such a functor is a comorphism of sites if and only if it has the cover lifting property. So the Giraud topology is nothing else than the topology generated by (the pullback-stable family of) sieves $\{f:\dom(f)\xrightarrow{}(c,x)\,|\,\pi_\dcat(f)\in S\}$ for every covering sieve $S$ on $c$. However, we have a factorization result on $\mathcal{G}(\dcat)$ which makes it even easier. Indeed, any morphism $(f,\alpha):(d,y)\xrightarrow{}(c,x)$ can clearly be written as $$(d,y)\xrightarrow{(\id,\alpha)}(d,\dcat(f)(x))\xrightarrow{(f,\id)}(c,x)$$
Thus, the sieves that we mentioned above are the sieves generated by the families $\{(f,\id):(\dom(f),\dcat(f)(x))\xrightarrow{}(c,x)\,|\,f\in S\}$ for covering sieves $S$ on $c$. So equivalently, a relative site is a functor $\mathcal{G}(\dcat)\xrightarrow{}C$ and a topology $K$ on $\mathcal{G}(\dcat)$ which contains the Giraud topology over $J$. We will use the Giraud topology later. Now we can define internal locales. 
\begin{definition}
    Let $\mathbf{Sh}(\mathcal{C},J)$ be any topos, $(\mathcal{C},J)$ a cartesian site. An internal locale $L$ is an object of $\mathbf{Sh}(\mathcal{C},J)$ which is pointwise a frame, such that the transition maps $L(f)$ are frame homomorphisms with a left adjoint $\exists_f\dashv L(f)$, and satisfying the two following conditions:\\
    \begin{enumerate}[(i)]
        \item (Beck-Chevalley) For any pullback square in $\mathcal{C}$
        \begin{center}
    \begin{tikzcd}
    a\arrow[r,"f"]\arrow[d,"g"'] & \arrow[d,"e"] b\\
    c\arrow[r,"h"'] & d
    \end{tikzcd}
    \end{center}
the following square 
\begin{center}
    \begin{tikzcd}
    L(a)\arrow[d,"\exists_g"'] & \arrow[l,"L(f)"']\arrow[d,"\exists_e"] L(b)\\
    L(c) & L(d) \arrow[l,"L(h)"]
    \end{tikzcd}
    \end{center}
 commutes;\\
 \item (Frobenius reciprocity) for any $f:d\xrightarrow{}c$, $l\in L(d)$ and $l'\in L(c)$, we have 
$$\exists_f(L(f)(l') \land l) = \exists_f(l) \land l'$$
    \end{enumerate}
\end{definition}

This definition exists already in \cite{johnstone2002sketches} section C1.6, but the reader should know that this is just a particular case of a more general kind of objects, known as \emph{existential fibred sites}, introduced in \cite{caramello2022fibred}. In fact, by Proposition~5.9 therein, an object of $\mathbf{Sh}(\mathcal{C},J)$ which is pointwise a frame and whose transition maps are frame homomorphisms with left adjoints is an internal locale if and only it is existential when considered as a fibred site. In what follows, we will not specify to which frame $L(c)$ refers some frame operation when it can be deduced from the `types'. For example $1$ will indifferently designate the top object of any $L(c)$, when we can clearly see to what $c$ we are referring.

\begin{definition}
    Let $L$ be an internal locale in $\mathbf{Sh}(\mathcal{C},J)$. The existential topology on $\mathcal{G}(L)$ is defined as follows: a sieve $S$ on $(c,x)$ is covering if and only if it contains a family $\{f_i:(c_i,x_i)\xrightarrow{}(c,x)\,|\,i\in I\}$ such that $\bigvee_{i\in I}\exists_{f_i}(x_i)=x$. We note it $J_L^{ext}$ or $J^{ext}$.
\end{definition}

In fact it is not obvious that this indeed defines (and not generates) a topology. But by Theorem~5.1 of \cite{caramello2022fibred} (applied to internal locales), an object $L$ of $\mathbf{Sh}(\mathcal{C},J)$ which is pointwise a frame and whose transition maps are frame homomorphisms with left adjoints is an internal locale if and only the existential topology on $\mathcal{G}(L)$ is indeed a Gothendieck topology (specifically, is stable under pullbacks). Notice that in the case of internal locales, we just write $f:(d,y)\xrightarrow{}(c,x)$ to mean $(f,*):(d,y)\xrightarrow{}(c,x)$ when $y\leq L(f)(x)$. We also have the following proposition, for which we give a direct proof:

\begin{proposition}[\cite{caramello2022fibred}]
    For an internal locale $L$ in $\mathbf{Sh}(\mathcal{C},J)$, the existential topology on $\mathcal{G}(L)$ contains the Giraud topology over $J$.
\end{proposition}
\begin{proof}
    We have to show that for all $x\in L(c)$, for every covering family $f_i:c_i\xrightarrow{}c$, $i\in I$, we have $x=\bigvee_i \exists_{f_i}(L(f_i)(x))$. Now we have 
    $$\bigvee_i \exists_{f_i}(L(f_i)(x))=\bigvee_i \exists_{f_i}(L(f_i)(x)\land 1)=\bigvee_i (\exists_{f_i}(1)\land x)=x\land \bigvee_i \exists_{f_i}(1)$$
    by Frobenius, so we just have to prove that $\bigvee_i \exists_{f_i}(1)=1$. Now starting from $\exists_{f_j}(1)\leq \bigvee_{i\in I}\exists_{f_i}(1)$, we get that $1\leq L(f_j)(\bigvee_{i\in I}\exists_{f_i}(1))$ for all $j$, i.e. $L(f_j)(\bigvee_{i\in I}\exists_{f_i}(1))=1=L(f_j)(1)$ for all $j$. Now clearly the $L(f_j)(\bigvee_{i\in I}\exists_{f_i}(1))$ form a matching family, so since $L$ is a sheaf and $f_i$ is covering, the uniqueness of the amalgamation gives that $\bigvee_{i\in I}\exists_{f_i}(1)=1$.
\end{proof}

Now last Proposition~shows that for an internal locale $L$ in $\mathbf{Sh}(\mathcal{C},J)$, $\pi_L:(\mathcal{G}(L),J^{ext})\xrightarrow{}(\mathcal{C},J)$ is a comorphism of sites, so it induces a geometric morphism $C_{\pi_L}:\mathbf{Sh}(\mathcal{G}(L),J^{ext})\xrightarrow{}\mathbf{Sh}(\mathcal{C},J)$. The topos $\mathbf{Sh}(\mathcal{G}(L),J^{ext})$ is called the \emph{topos of internal sheaves} on $L$, sometimes noted $\mathbf{Sh}_{\mathbf{Sh}(\mathcal{C},J)}(L)$. In fact, since $\pi_L$ is faithful, by Proposition~7.11 \cite{caramello2019denseness}, we have that $C_{\pi_L}$ is $\emph{localic}$.\\
\\
Now we can state the result that will be the most important for us:

\begin{proposition}[\cite{caramello2022fibred} Proposition~4.2]\label{int_loc_eq}
    For a topos $\mathcal{E}=\mathbf{Sh}(\mathcal{C},J)$, there is an equivalence
$$Loc(\mathcal{E}) \cong \textsc{Loc}/\mathcal{E}$$
where $Loc(\mathcal{E})$ is the category of internal locales in $\mathcal{E}$ and \textsc{Loc} is the category of toposes and localic morphisms
between them. The equivalence sends an internal locale $L$ to $C_{\pi_L}$ and a localic geometric morphism $f: \mathcal{F} \xrightarrow{} \mathcal{E}$ to the functor $L_f: \mathcal{C}^{op} \xrightarrow{} \Set$ sending an object $c$ of $\mathcal{C}$ to the frame
$\textup{Sub}_\mathcal{F}(f^*(l(c)))$, which yields an internal locale in $\mathcal{E}$.
\end{proposition}

This means in particular that every internal locale is of the form $L_f$ for some localic geometric morphism $f$. Note that we did not really define what is an internal locale morphism. We have the usual (diagrammatic) internal definition, but an external definition can be found in \cite{wrigley2023some}. \\
\\
The previous Proposition~enables us to generalize the well-known operation of `ideal completion'. Indeed, given a distributive lattice $P$, we can consider the lattice of ideals of $P$, which will in fact be a frame. More generally, given a preorder $P$ seen as a category and a coverage $J$ on $P$, we can consider the frame of $J$-ideals of $P$, where a $J$-ideal is a subset $I$ of $P$ which is a lowerset and satisfies the following property: for every $x\in P$ and covering family $S$ on $x$, if for every $x_i\in S$ we have $x_i\in I$ then $x\in I$. For example, if $P$ is a distributive lattice, taking $J(x)$ to be the coverage generated by finite families $(x_i)_i$ with $\bigvee_i x_i=x$, we get back the usual ideal completion, i.e. a $J$-ideal is just an ideal in the usual sense. \\
\\
Now, as developped in section 6 of \cite{caramello2022fibred}, the last Proposition~enables us to generalize this to the case of internal preoders. For a site $(\mathcal{C},J),$ define a \emph{fibred preoder} $\pcat$ to be a $\mathcal{C}$-indexed category which is pointwise a preorder. For any topology $K$ on $\mathcal{G}(\pcat)$ containing the Giraud topology over $J$, the morphism $C_{\pi_\pcat}$ is localic like before, so it corresponds by the Proposition~to an internal locale $L_\pcat^K$ of $\mathbf{Sh}(\mathcal{C},J)$. This is what we call the \emph{fibred ideal completion} of $\pcat$ relative to $K$. Notice that in this case, $\pi_\pcat$ has left and right adjoints $c\mapsto (c,0)$ and $c\mapsto (c,1)$, and by Proposition~3.14 of \cite{caramello2019denseness}, the latter is a morphism of sites $i_\pcat:(\mathcal{C},J)\xrightarrow{}(\mathcal{G}(\pcat),K)$ such that $\mathbf{Sh}(i_\pcat)=C_{\pi_\pcat}$. Now we have an explicit definition
\begin{align*}
    L_\pcat^K(c) & =\textup{Sub}_{\mathbf{Sh}(\mathcal{G}(\pcat),K)}(C_{\pi_\pcat}^*(l_J(c)))\\
    & =\textup{Sub}_{\mathbf{Sh}(\mathcal{G}(\pcat),K)}(l_K((c,1)))\quad \text{(by definition of induced morphism)}\\
    & =\textup{ClSub}^K_{\mathbf{Psh}(\mathcal{G}(\pcat))}(\Hom_{\mathcal{G}(\pcat)}(-,(c,1)))\quad\text{(by Proposition~2.3 of \cite{caramello2019denseness})}\\
    & =\textup{ClSub}^K_{\mathbf{Psh}(\mathcal{G}(\pcat))}(\Hom_C(\pi_\pcat(-),c)) \quad\text{(by adjointness)}
\end{align*}
where $l_J$ is the Yoneda embedding followed by sheafification, and $$\textup{ClSub}^K_{\mathbf{Psh}(\mathcal{G}(\pcat))}(\Hom_{\mathcal{G}(\pcat)}(-,(c,1)))$$ is the set a $K$-closed subobjects of the presheaf $\Hom_{\mathcal{G}(\pcat)}(-,(c,1))$, i.e. the set of $K$-closed sieve on $(c,1)$.\\
\\
Finally, note that we have a canonical morphism $\pcat\xrightarrow{}L_\pcat^K$. Indeed, for an object $c$ of $\mathcal{C}$, for an element $x\in\pcat(c)$, we can consider the $K$-closure of the sieve generated by $\id:(c,x)\xrightarrow{}(c,1)$. The details are easily checked.

\subsubsection{Gleason's cover, explicitly}
Using the tools that we just exposed, we give an explicit description of Gleason's cover for toposes. The problem is the following: given a topos $\mathcal{E}$, we want to build a topos $\gamma\mathcal{E}$ satisfying De Morgan's law and a surjection $\gamma\mathcal{E}\xrightarrow{}\mathcal{E}$ which is minimal in an appropriate sense. The essential results of this section are not new and can be found in \cite{johnstone1980gleason} and \cite{harun1996applications}, but we give new proofs using explicit calculations and not internal arguments.

\begin{definition}
    Let $A$ be an internal distributive lattice in a topos $\mathcal{E}=\mathbf{Sh}(\mathcal{C},J)$. The coherent coverage $K^{(A)}$ on $A$ is the coverage defined on the site $\mathcal{G}(A)$ by the following two (pullback stable) types of families of arrows:
    \begin{enumerate}[(i)]
        \item The families of the form $\{(c,x_i)\xrightarrow{id}(c,x)\,|\,i\in I\}$ for $I$ finite and $x=\bigvee_i x_i$, which we call $K_1^{(A)}$;
    \item The families of the form $\{(c_i,A(f_i)(x))\xrightarrow{f_i}(c,x)\,|\,i\in I\}$ for $\{f_i\,|\,i\in I\}\in J(c)$, called $K_2^{(A)}$.
    \end{enumerate}
\end{definition}

We will write $K$ when the lattice we are talking about is clear. Notice that $K_2$ exactly generates the Giraud topology, so we have a comorphism of sites $\pi_A:(\mathcal{G}(A),K)\xrightarrow{}(\mathcal{C},J)$, having right adjoint $i_A:c\mapsto (c,1)$, so we have a generated geometric morphism $\mathbf{Sh}(i_A)=C_{\pi_A}$. Now recall the canonical internal Boolean algebra $\Omega_{\neg\neg}$.

\begin{definition}
    Gleason's cover for a topos $\mathcal{E}$ is defined by the generated morphism $C_{\pi_{\Omega_{\neg\neg}}}:\mathbf{Sh}(\mathcal{G}(\Omega_{\neg\neg}),K)\xrightarrow{}\mathcal{E}$ (or equivalently as $\mathbf{Sh}(i_{\Omega_{\neg\neg}})$).
\end{definition}

Since $A$ above is a fibred preoder, we can considered its fibred ideal completion with respect to $K$, which is an internal locale. We just call this locale the ideal completion of $A$, and we note it $\textup{Idl}(A)$. Gleason's cover will then be equivalent to the morphism generated by $\pi_{\textup{Idl}(\Omega_{\neg\neg})}:(\mathcal{G}(\textup{Idl}(\Omega_{\neg\neg})),J^{ext})\xrightarrow{}(\mathcal{C},J)$.\\
\\
Note that this definition is equivalent to the one in \cite{johnstone1980gleason}, as proved there. However, we will prove all the important properties of Gleason's cover without any reference to Johnstone's definition. We want to avoid doing things by internalizing known theorems on locales, to get more enlightening proofs. Our first goal is thus to show that Gleason's covers satisfies De Morgan's law.

\begin{proposition}\label{reldml}
    Let $L$ be an internal locale in $\mathcal{E}$. Then $\mathbf{Sh}(\mathcal{G}(L),J^{ext})$ satisfies De Morgan's law if and only if $L$ satisfies De Morgan's law (i.e. is a Stone algebra).
\end{proposition}
\begin{proof}
    Let us right $f:\mathbf{Sh}(\mathcal{G}(L),J^{ext})\xrightarrow{}\mathcal{E}$ for the canonical morphism. We first prove that $L=f_*(\Omega_{\mathbf{Sh}(\mathcal{G}(L),J^{ext})})$, i.e. we want to prove that for every object $c$ of $\mathcal{C}$, we have $L(c)=\Omega((c,1))$ (natural in $c$). Thus, it suffices to prove that every $J^{ext}$-closed sieve on $(c,1)$ is principal. Consider such sieve $S$, and write $t:=\bigvee_{\substack{h:(d,x)\xrightarrow{}(c,1)\\h\in S}}\exists_h(x)$. Now it is clear that $h\in S$ if and only if $h$ factors through $(c,t)$, so we just need to notice that $(c,t)\xrightarrow{id}(c,1)$ is in $S$ because $S$ is $J^{ext}$-closed and $\{\Tilde{h}:(d,x)\xrightarrow{}(c,t)\,|\,h\in S\}$ is $J^{ext}$-covering by definition. Now to conclude we still have to see that for any $x\in L(c)$, $(c,x)\xrightarrow{id}(c,1)$ is a monomorphism, so $\exists_{id}:\Omega((c,x))\xrightarrow{}\Omega((c,1))$ is just given by composition and is injective, and so $\Omega((c,x))=\Omega((c,1))_{\leq a\yo((c,x))}=L(c)_{\leq x}$. Now $\mathbf{Sh}(\mathcal{G}(L),J^{ext})$ satisfies De Morgan's law if and only if for every $(c,x)$, $\Omega((c,x))$ satisfies De Morgan's law if and only if for every $(c,x)$, $L(c)_{\leq x}$ satisfies De Morgan's law if and only if for every $c$, $L(c)$ satisfies De Morgan's law.
\end{proof}

 Considering our goal to prove that Gleason's cover satisfies De Morgan's law, we still need to prove that $\textup{Idl}(\Omega_{\neg\neg})$ is a Stone algebra. We will use the fact that $\Omega_{\neg\neg}$ is an internal complete Boolean algebra, i.e. an internal locale which is also an internal Boolean algebra. But first let us look closer at the topology generated by $K$. Consider an object $(c,x)$ of $\mathcal{G}(B)$, with $B$ an internal locale. Since the existential topology on an internal locale contains the Giraud topology, we see that $K\subseteq J^{ext}$. We then easily conclude that $\varnothing\in K((c,x))$ if and only if $x=0$. We also need a technical lemma.

 \begin{Lemma}
     If $B$ is an internal locale, then the transition functions preserve pseudo-complements when they exist.
 \end{Lemma}
 \begin{proof}
     \begin{align*}
         y\leq B(f)(\neg x) & \iff \exists_f(y)\leq \neg x\\
         & \iff \exists_f(y)\land x \leq 0\\
         & \iff \exists_f(y\land B(f)(x))\leq 0 \quad \text{(by Frobenius)}\\
         & \iff y\land B(f)(x)\leq 0\\
         & \iff y\leq \neg B(f)(x)
    \end{align*}
 \end{proof}

\begin{proposition}[\cite{johnstone1980gleason} Lemma~1.1]
    Let $B$ be an internal locale in $\mathcal{E}$ which is also a Boolean algebra. Then $\textup{Idl}(B)$ is a Stone algebra. 
\end{proposition}
\begin{proof}
    Take $\mathcal{E}=\mathbf{Sh}(\mathcal{C},J)$, $\mathcal{C}$ the geometric site of definition. As noted before, $\textup{Idl}(B)(c,1)$ is the algebra of $K$-closed sieves on $(c,1)$. We write $0_{d,x}$ for the smallest $K$-closed sieve on $(d,x)$, namely the sieve $\{(d',x')\xrightarrow{}(d,x)\,|\,\varnothing\in K((d',x'))\}=\{(d',x')\xrightarrow{}(d,x)\,|\,x'=0\}$ by the previous discussion. Now consider $R$ a $K$-closed sieves on $(c,1)$. Clearly the pseudo-complement in this algebra is given by $\neg R=\{f:(d,x)\xrightarrow{}(c,1)\,|\,f^*(R)=0_{d,x}\}$. We claim that $$\neg R = \{f:(d,x)\xrightarrow{}(c,1)\,|\,\forall g:(d',x')\xrightarrow{}(c,1)\in R,\,\exists_f(x)\land\exists_g(x')=0\}$$
    Indeed, take $f:(d,x)\xrightarrow{}(c,1)$ in $\neg R$, $g:(d',x')\xrightarrow{}(c,1)$ in R. We have the following pullback square:
\begin{center}
            \begin{tikzcd}
        (d\times_c d',B(\pi)(x)\land B(\pi')(x')) \arrow[r,"\pi"]\arrow[d,"\pi'"] & (d,x) \arrow[d,"f"] \\
        (d',x')\arrow[r,"g"] & (c,1)
    \end{tikzcd}
    \end{center}
By definition of $\neg R$, we have $B(\pi)(x)\land B(\pi')(x')=0$. Now, the following algebraic manipulation shows that this is equivalent to $\exists_f(x)\land\exists_g(x')=0$:
\begin{align*}
    B(\pi)(x)\land B(\pi')(x')=0 & \iff B(\pi)(x)\leq B(\pi')(\neg x') \\
    & \iff \exists_\pi(B(\pi')(x'))\leq\neg x\\
    & \iff B(f)(\exists_g(x'))\leq\neg x \text{    (by Beck-Chevalley)}\\
    & \iff x\land B(f)(\exists_g(x')) = 0\\
    & \iff x\leq B(f)(\neg \exists_g(x'))\\
    & \iff \exists_f(x)\leq\neg\exists_g(x')\\
    & \iff \exists_f(x)\land\exists_g(x')=0
\end{align*}

Conversely, take $f$ such that for all $g:(d',x')\xrightarrow{}(c,1)$ in $R$, $\exists_f(x)\land\exists_g(x')=0$. Now take $g:(d',x')\xrightarrow{}(d,x)$ such that $f\circ g$ is in $R$. We have to show that $x'=0$. By definition of morphisms in $\mathcal{G}(B)$, we have that $\exists_g(x')\leq x$, and so $\exists_{f\circ g}(x')\leq\exists_f(x)$. So $\exists_{f\circ g}(x')=\exists_{f\circ g}(x')\land \exists_f(x)$, but since $f\circ g\in R$, by hypothesis we have that this is $0$. We conclude that $x'=0$. This proves the claim.\\
\\
As a consequence of the claim, if $f:(d,x)\xrightarrow{}(c,1)$ is in $\neg R$, then $\id:(c,\exists_f(x))\xrightarrow{}(c,1)$ is also in $\neg R$. So $\neg R$ is the sieve generated by the $\id:(c,\exists_h(x))\xrightarrow{}(c,1)$, $h\in\neg R$. In fact, it is the principal sieve generated by $(c,\bigvee_{\substack{h:(d,x)\xrightarrow{}(c,1)\\h\in\neg R}}\exists_h(x))$, by the distributivity law. Now consider the principal sieve generated by $(c,\neg\bigvee_{\substack{h:(d,x)\xrightarrow{}(c,1)\\h\in\neg R}}\exists_h(x))$, noted $R'$. Let $\mathcal{C}l_K(R')$ be its $K$-closure. Clearly $\neg R\cap R'\subseteq 0_{c,1}$, so $\neg R\land Cl_K(R')=Cl_K(\neg R\cap R')= 0_{c,1}$. Also notice that for a sieve $S$ on $(c,1)$, $\mathcal{C}l_K(S)=t_{(c,1)}$ the maximal sieve if and only if $\id\in Cl_K(S)$ if and only if $S$ is $J$-covering. So since $\neg R\cup R'$ $K$-covers $(c,1)$, $\neg R \lor Cl_K(R')=Cl_K(\neg R\cup Cl_K(R'))=t_{(c,1)}$. So $\neg R$ is complemented, which concludes the proof. 
\end{proof}

\begin{remark}
    We have a converse for the previous proposition. Indeed, if $B$ is only supposed to be a Boolean algebra, and $\textup{Idl}(B)$ is supposed to be a Stone algebra, then the obvious morphism $B\xrightarrow{}\textup{Idl}(B)$ which maps $x\in B(c)$ to the principal sieve generated by $(c,x)$ identifies each $B(c)$ with the principal sieves on $(c,1)$, thus the complemented ones as we just saw, equivalently the $\neg\neg$-closed ones by `Stone algebra', i.e. it identifies $B$ with $(\textup{Idl}(B))_{\neg\neg}$. But it is not difficult to check that $(\textup{Idl}(B))_{\neg\neg}$ inherits the internal locale structure of $\textup{Idl}(B)$, so $B$ is an internal locale.
    \end{remark} 

\begin{remark}\label{dml_ideals}
    In fact, $B$ does not have to be Boolean. Being a Stone algebra is enough, since we can easily show that starting from a $K$-closed sieve $R$ as in the proof of the proposition, we get $\neg R=((c,z))$ the sieve generated by $\id:(c,z)\xrightarrow{}(c,1)$ for a specific $z$, then using the same characterization of pseudo-complements of sieves, we get that $\neg \neg R=((c,\neg z))$, so $\neg R=\neg\neg\neg R=((c,\neg \neg z))$, and so since $B$ is a Stone algebra, we get that $\neg R\lor \neg\neg R=1$. We can also directly show that $z$ is $\neg\neg$-closed, by showing that $\exists_f(x)\land\exists_g(x')=0$ if and only if $\exists_f(\neg\neg x)\land\exists_g(x')=0$, by forming the pullback of $f$ and $g$ then using the equivalent formulation with the right adjoint and the fact that it preserves pseudo-complements. This leads us to ask the following question: what happens if we replace $\Omega_{\neg\neg}$ in the Gleason cover construction by the DeMorganization $\Omega_m$ of $\Omega$ \cite{caramello2009morgan}? We answer this question at the end of section \ref{gc_localic}, by showing that it does not satisfy one of the key properties of Gleason's cover.
\end{remark}

\begin{corollary}[\cite{johnstone1980gleason} Corollary~1.2]
    Gleason's cover satisfies De Morgan's law.
\end{corollary}
\begin{proof}
    Combine the previous two propositions.
\end{proof}

Now we want the surjectivity of the canonical morphism $\gamma\mathcal{E}\xrightarrow{}\mathcal{E}$. We say that an internal locale $L$ is \emph{non trivial} if $0\ne 1$ in $L(c)$ for every $c$ which is not covered by $\varnothing$. 

\begin{proposition}\label{surj_gc}
    Let $A$ be an internal locale in $\mathcal{E}$. Then $\mathbf{Sh}(i_A):\mathbf{Sh}(\mathcal{G}(A),K)\xrightarrow{}\mathcal{E}$, with $K$ the coherent coverage, is surjective if and only if $A$ is non trivial.
\end{proposition}
\begin{proof}
    Use \cite{caramello2019denseness} Theorem~6.3(i) and the fact that a nonempty $K$-covering family always projects to a $J$-covering family, the case of $K_1$ families being trivial.
\end{proof}

From this we immediately get what we wanted. 

\begin{corollary}[\cite{johnstone1980gleason} Corollary~1.4]\label{gc_surj}
$\gamma\mathcal{E}\xrightarrow{}\mathcal{E}$ is surjective.
\end{corollary}

Now we want to show that if we start with a De Morgan topos and take its Gleason cover, then we get the topos itself, i.e. that Gleason's cover `only involves De Morgan's law' in some sense. For this we need some lemmas.

\begin{Lemma}\label{1u1}
    Let $B:=1\sqcup 1$ in $\mathcal{E}$. Then $\textup{Idl}(B)=\Omega$.
\end{Lemma}
\begin{proof}
We want a natural isomorphism between $K$-closed sieves on $(c,1)$ and $J$-closed sieves on c. We first consider the subobject $1$ of $\textup{Idl}(B)$, which is the maximal sieve pointwise. This subobject has a classifying map, let us call it $\rho$. Our ultimate goal is to prove that this $\rho$ is an isomorphism. First, we want to describe it explicitly.  It is not hard to check from the definition that for a $K$-closed sieve $R$ on $(c,1)$ and a morphism $f:d\xrightarrow{}c$, we have $\textup{Idl}(B)(f)(R)=f^*(R)$, where we see $f$ as a morphism $(d,1)\xrightarrow{}(c,1)$. So $\rho_c(R)=\{f:d\xrightarrow{}c\,|\,f^*(R)=t_{(c,1)}\}=\{f:d\xrightarrow{}c\,|\,(d,1)\xrightarrow[]{f}(c,1)\in R\}$. In particular it is a $J$-closed sieve. \\
\\
Now let us prove injectivity: we have to prove that $R$ is completely determined by its morphisms of the form $(d,1)\xrightarrow{}(c,1)$. We call these the underlying $J$-closed sieve of $R$, and note it $S_R$. Consider a morphism $f:(d,x)\xrightarrow{}(c,1)$. Here $x$ is an element of $B(d)$, i.e. a complemented $J$-closed sieve on $d$. We claim that $f\in R\iff x\subseteq f^*(S_R)$. Indeed, take $f:(d,x)\xrightarrow{}(c,1)$ in $R$, $g:d'\xrightarrow{}d$ in $x$, we have a morphism $(d',1)=(d',g^*(x))\xrightarrow{g}(d,x)$, so $(d',1)\xrightarrow{g}(d,x)\xrightarrow{f}(c,1)$ is in $R$, and so $f\circ g\in S_R$. Conversely, suppose that $x\subseteq f^*(S_R)$, and let $\neg x$ the complement of $x$ as a $J$-closed sieve, i.e. $x\cup\neg x$ is $J$-covering. We want to show that $f\in R$, but since $R$ is $K$-closed, it is enough to show that $(d',g^*(x))\xrightarrow{g}(d,x)\xrightarrow{f}(c,1)\in R$ for all $g\in x\cup\neg x$. If $g\in x$, by hypothesis we have $f\circ g\in S_R$, i.e. $(d',g^*(x))=(d',1)\xrightarrow{f\circ g}(c,1)\in R$, and if $g\in\neg x$ then $g^*(x)=0$, so $(d',g^*(x))$ is covered by $\varnothing$, and so any morphism from $(d',g^*(x))$ to $(c,1)$ belongs to the $K$-closed sieve $R$. This concludes the proof of the claim and of the injectivity. \\
\\
Finally for surjectivity, take a $J$-closed sieve $S$ on $c$. We just need to prove that $\{f:(d,x)\xrightarrow{}(c,1)\,|\,x\subseteq f^*(S)\}$ is a $K$-closed sieve, because this operation does not add additional morphisms when $x=t_c$, i.e. $x=1$. It is clearly a sieve. First, let us prove that it is $K_2$-closed: suppose that $S'=\{f_i:d_i\xrightarrow{}d\,|\,i\in I\}$ $J$-covers $d$, and that $f_i^*(x)\subseteq(f\circ f_i)^*(S)$ for all $i\in I$. Now take $g:d'\xrightarrow{}d$ in $x$, we want to prove that $f\circ g\in S$. We have $g^*(S')\in J(d')$, so by $J$-closeness of $S$ it is enough to prove that $f\circ g\circ f'\in S$ for every $f'\in g^*(S')$. But by definition $g\circ f'\in S'$, so we have a commutative square:
\begin{center}
    \begin{tikzcd}
    d'' \arrow[r,"f'"]\arrow[d,"g'"'] & d' \arrow[d,"g"] \\
    d_i\arrow[r,"f_i"'] & d
    \end{tikzcd}
\end{center}
for some $i$. So $g'\in f_i^*(x)\subseteq(f\circ f_i)^*(S)$, so $f\circ f_i\circ g'\in S$, so $f\circ g\circ f'\in S$. So to finish the proof we just need to prove $K_1$-closeness. Clearly it is enough to prove that  the join in $B$ is computed as in $\Omega$. In other words, we have to prove that if $x,\,y\in \Omega(c)$ are complemented, then $x\lor y$ is complemented. This follows from:
\begin{align*}
    1 & =((x\lor\neg x)\land y)\lor ((x\lor\neg x)\land \neg y) \\
    & = (x\land y)\lor(\neg x\land y)\lor(x\land \neg y)\lor(\neg x\land \neg y)\\
    & \leq x\lor y\lor  (\neg x\land \neg y)\\
    & =(x\lor y)\lor\neg (x\lor y)
\end{align*}

\end{proof}

\begin{Lemma}\label{triv}
    $C_{\pi_\Omega}\cong \textup{Id}$ as relative toposes over $\mathcal{E}$.
\end{Lemma}
\begin{proof}
    Clearly $\Omega\cong L_{\textup{Id}}$, with the notations of Proposition~\ref{int_loc_eq}. Then just apply this proposition.
\end{proof}

\begin{corollary}[\cite{johnstone1980gleason} Corollary~1.5]
$\gamma\mathcal{E}\xrightarrow{}\mathcal{E}$ is an equivalence if and only if $\mathcal{E} $ satisfies De Morgan's law.
\end{corollary}
\begin{proof}
    If $\mathcal{E}$ satisfies De Morgan's law, then $1\sqcup 1=\Omega_{\neg\neg}$, so combining the previous lemmas $\gamma\mathcal{E}\xrightarrow{}\mathcal{E}$ is an equivalence. The other direction just follows from the fact that  $\gamma\mathcal{E}$ satisfies De Morgan's law.
\end{proof}

Another important characteristic of Gleason's cover is minimality.

\begin{definition}
    An internal locale $A$ in $\mathcal{E}$ is minimal if the following diagram is a pullback:
    \begin{center}
    \begin{tikzcd}
    1 \arrow[r,""]\arrow[d,"0_A"'] & 1 \arrow[d,"\bot"] \\
    A\arrow[r,"\rho"'] & \Omega
    \end{tikzcd}
\end{center}
where $\rho$ is the classifying map of the subobject $1\xrightarrow{1_A}A$.
\end{definition}
Minimality means that for every object $c$, the only element $x$ of $A(c)$ such that for all $f:d\xrightarrow{}c$ with $d$ not covered by $\varnothing$ we have $A(f)(x)\ne 1$ is 0. We call this the `characterization of 0 at $c$'. Notice that a minimal locale is always non trivial, because if it is not for some $c$, then $*=0=1$, so 0 does not even satisfy the characterization of 0 at $c$, i.e. the diagram in the definition of minimality is not even commutative.

\begin{Lemma}[\cite{harun1996applications} Lemma~3.3]
    With the previous notations, if $A$ is minimal then $\rho$ preserves pseudo-complements.
\end{Lemma}
\begin{proof}
    Suppose $A$ minimal. We just compute:
    $$\rho_c(\neg a)=\{f:d\xrightarrow{}c\,|\,A(f)(a)=0\}$$
    $$\neg\rho_c(a)=\{f:d\xrightarrow{}c\,|\,\forall g:d'\xrightarrow{}d,\,\varnothing\in J(d')\lor A(f\circ g)(a)\ne 1\}$$
    If $f\in \rho_c(\neg a)$, $g:d'\xrightarrow{}d$ such that $A(f\circ g)=1$, then $0=A(g)(A(f)(a))=1$, so  $\varnothing\in J(d')$ because $A$ non trivial. Conversely, $f\in \neg\rho_c(a)$, we just use the characterization of 0 at $\dom(f)$ to deduce that $A(f)(a)=0$.
\end{proof}

\begin{remark}
    It is not hard to see that the converse of the last Lemma~is true for a non trivial locale.
\end{remark}

\begin{proposition}[\cite{harun1996applications} Lemma~3.4]\label{min_neg}
    If $A$ is minimal, then $\rho$ maps $A_{\neg\neg}$ isomorphically onto $\Omega_{\neg\neg}$.
\end{proposition}
\begin{proof}
    We work with a cartesian site of definition. Using the previous proposition, $A_{\neg\neg}$ is indeed mapped to $\Omega_{\neg\neg}$ by $\rho$. Now if $a\in A(c)$ and $\rho_c(a)=0$, then $a=0$ by minimality. So if $\rho_c(x)=\rho_c(y)$, then equivalently $\rho_c(x\land\neg y)=0$ and $\rho_c(y\land\neg x)=0$ since $\rho_c$ commutes with finite meets and pseudo-complements. So  $x\land\neg y=y\land\neg x=0$, i.e. $x=y$. This proves injectivity. For surjectivity, notice that we can define a morphism $\lambda:\Omega\xrightarrow{}A$ by $\lambda_c(S)=\bigvee_{f\in S}\exists_f(1)$. We have to check naturality, i.e. for $h:d\xrightarrow{}c$, we should have $\bigvee_{f\in S}A(h)(\exists_f(1))=\bigvee_{f\in h^*(S)}\exists_f(1)$. To prove $\bigvee_{f\in S}A(h)(\exists_f(1))\leq\bigvee_{f\in h^*(S)}\exists_f(1)$, consider $f\in S$ and take the pullback:
    \begin{center}
    \begin{tikzcd}
    d\times_c d' \arrow[r,"\pi"]\arrow[d,"\pi'"'] & d \arrow[d,"h"] \\
    d'\arrow[r,"f"'] & c
    \end{tikzcd}
\end{center}
    By Beck-Chevalley, $A(h)(\exists_f(1))=\exists_\pi( A(\pi')(1))=\exists_\pi(1)$ with $\pi\in h^*(S)$. To prove the converse, just notice that for $h\circ f\in S$, we have $\exists_f(1)\leq A(h)(\exists_h\exists_f(1))$. This concludes the proof of naturality of $\lambda$. In addition we can see that we  have $S\subseteq\rho_c\lambda_c(S)$, and also $\lambda_c(\neg S)\leq \neg \lambda_c (S)$. The first one is easy, for the second notice that this is equivalent to $\exists_f(1)\land\exists_g(1)=0$ for all $f\in\neg S$ and $g\in S$. But in this case we have the following pullback square:
\begin{center}
    \begin{tikzcd}
    0 \arrow[r,"\pi'"]\arrow[d,"\pi"'] & d' \arrow[d,"g"] \\
    d\arrow[r,"f"'] & c
    \end{tikzcd}
\end{center}
So trivially $\exists_\pi(A(\pi')(1))=0$, i.e. by Beck-Chevalley $A(f)(\exists_g(1))=0$. But by Frobenius we have $\exists_g(1)\land\exists_f(1)=\exists_f(A(f)(\exists_g(1))\land 1)=0$.
    Now $\rho_c(\neg\neg\lambda_c(S))\land \neg S\leq\rho_c(\neg\neg\lambda_c(S))\land \rho_c\lambda_c(\neg S)=0$, and conversely $S\leq \rho_c\lambda_c(S)\leq \rho_c(\neg\neg\lambda_c(S))$ so $\neg\neg S\leq \rho_c(\neg\neg\lambda_c(S))$.
\end{proof}

So we have a commuting diagram:
\begin{center}
    \begin{tikzcd}
    A_{\neg\neg} \arrow[r,"\cong"]\arrow[d,""'] & \Omega_{\neg\neg} \arrow[d,""] \\
    A\arrow[r,"\rho"'] & \Omega
    \end{tikzcd}
\end{center}

Now we make the link with Gleason's cover.
\begin{proposition}
    $\textup{Idl}(\Omega_{\neg\neg})$ is minimal.
\end{proposition}
\begin{proof}
    Recall the explicit description of $\rho$ in this case from the proof of \ref{1u1}: $\rho_c(R)$ is the sieve of morphisms of $R$ of the form $(d,1)\xrightarrow{}(c,1)$. Now just notice that if $R\ne 0$, then there is $f:(d,x)\xrightarrow{}(c,1)$ in $R$ with $x\in\Omega_{\neg\neg}(d)$, $x\ne 0$, so take any $g:d'\xrightarrow{}d$ in $x$ with $\varnothing\not\in J(d')$, so $(d',1)=(d',g^*(x))\xrightarrow{g}(d,x)\xrightarrow{f}(c,1)$ in $R$, so $f\circ g\in\rho_c(R)$, so $\rho_c(R)\ne 0$.
\end{proof}
As a corollary, by Proposition~\ref{min_neg} we have the commuting diagram:
\begin{center}
    \begin{tikzcd}
    (\textup{Idl}(\Omega_{\neg\neg}))_{\neg\neg} \arrow[r,"\cong"]\arrow[d,""'] & \Omega_{\neg\neg} \arrow[d,""] \\
    (\textup{Idl}(\Omega_{\neg\neg}))\arrow[r,"\rho"'] & \Omega
    \end{tikzcd}
\end{center}
This diagram will have an important consequence, but we need a Lemma~before.

\begin{Lemma}
    $\mathbf{Sh}(\mathcal{G}(A_{\neg\neg}),J^{ext})=(\mathbf{Sh}(\mathcal{G}(A),J^{ext}))_{\neg\neg}$.
\end{Lemma}
\begin{proof}
    By Proposition~\ref{int_loc_eq}, $A=L_f$ for $f:\mathbf{Sh}(\mathcal{G}(A),J^{ext})\xrightarrow{}\mathcal{E}$ the canonical morphism. But $$(L_f)_{\neg\neg}(E)=ClSub^{\neg\neg}_{\mathbf{Sh}(\mathcal{G}(A),J^{ext})}(f^*(E))=\textup{Sub}_{(\mathbf{Sh}(\mathcal{G}(A),J^{ext}))_{\neg\neg}}(a_{\neg\neg}(f^*(E)))$$
    with $i:(\mathbf{Sh}(\mathcal{G}(A),J^{ext}))_{\neg\neg}\hookrightarrow{}\mathbf{Sh}(\mathcal{G}(A),J^{ext})$. So $A_{\neg\neg}=L_{f\circ i}$, which concludes by \ref{int_loc_eq} again.  
\end{proof}

As a consequence, translating the previous diagram into a diagram of toposes, we get:
\begin{center}
    \begin{tikzcd}
    (\gamma\mathcal{E})_{\neg\neg}\arrow[r,"\cong"]\arrow[d,""'] & \arrow[d,""] \mathcal{E}_{\neg\neg}\\
    \gamma\mathcal{E}\arrow[r,""'] & \mathcal{E}
    \end{tikzcd}
\end{center}

This is what is needed in the applications of Gleason's cover in \cite{harun1996applications}. We can give a nice corollary.

\begin{corollary}[\cite{harun1996applications} Corollary~3.5]
    $\gamma\mathcal{E}$ is Boolean if and only if $\mathcal{E}$ is.
\end{corollary}
\begin{proof}
    If $\mathcal{E}$ is Boolean then it satisfies De Morgan's law so $\gamma\mathcal{E}\cong\mathcal{E}$. Conversely, if $ \gamma\mathcal{E}$ is Boolean, by the previous diagram $\mathcal{E}_{\neg\neg}\xrightarrow{}\mathcal{E}$ is surjective, and so being an inclusion it is an isomorphism.
\end{proof}

\subsubsection{Gleason's cover of a localic topos}\label{gc_localic}
Now we want to study the Gleason's cover construction for a topos $\mathbf{Sh}(L)$ for a locale $L$. As an application, we will get back Gleason's cover construction for compact Hausdorff spaces. The content of this section up to Corollary~\ref{gc_top}, including the proofs, is essentially the same as section 3 of \cite{johnstone1980gleason}, but we reformulate it in our framework for completeness.\\
\\
In fact, we will restrict our attention to regular locales: we say that a locale $L$ is \emph{regular} if for every $l\in L$, there are $l_i,\,t_i$ for $i\in I$ such that $l=\bigvee l_i$, $t_i\land l_i=0$ and $t_i\vee l= 1$ (notice that we may assume the $l_i$ to be $\neg\neg$-closed, because $t_i\land l_i=0$ if and only if $t_i\leq\neg l_i=\neg\neg\neg l_i$ if and only if $t_i\land\neg\neg l_i=0$). Regularity is a kind of separation property. A regular space $X$ is a space such that its locale of open sets $O(X)$ is regular. Note for example that a space is compact Hausdorff if and only if it is $T_0$ compact regular. Since a locale is always $T_0$, a compact Hausdorff locale is the same as a compact regular locale. It seems that this assumption is necessary to get a nice description of Gleason's cover in this case, but it is not completely satisfying because it is purely topological, and we would like a more general characterization. \\
\\
We say that a topos is \emph{localic} if it is equivalent to a topos $\mathbf{Sh}(L)$ for a locale $L$. Recall that this is the category of sheaves on the site $(L,J)$ where $(x_i)_i\in J(x)$ if and only if $\bigvee_i x_i=x$. Let us have a closer look at the subobject classifier in this case. Clearly, a $J$-closed sieve on $x\in L$ is an ideal of $L_{\leq x}$ which is closed under arbitrary join, so it is principal. In particular $\Omega_{\mathbf{Sh}(L)}(1)\cong L$ ($\cong \textup{Sub}_{\mathbf{Sh}(L)}(1)$ as mentioned before, because $l(1)=1$ where the first $1$ is the top element of $L$, i.e. its terminal object, and the second $1$ is the terminal object of $\mathbf{Sh}(L)$, and $l$ preserves finite limits). It is a general fact that that if $\mathcal{E}:=\mathbf{Sh}(L)$ is localic and $\gamma\mathcal{E}\xrightarrow{}\mathcal{E}$ is a localic geometric morphism, then $\gamma\mathcal{E}$ is localic, and in fact the corresponding locale is given by $\textup{Idl}(\Omega_{\neg\neg})(1)$. Now we need to give an explicit description of $\Omega_{\neg\neg}$. We know that $\Omega(l)=L_{\leq l}$, so $\Omega_{\neg\neg}(l)=(L_{\leq l})_{\neg\neg}$. But negation in $L_{\leq l}$ is given by $\neg^{L_{\leq l}}x=l\land\neg x$, so double negation is given by $l\land \neg(l\land\neg x)$. Now for any $y\leq l$, we have 
\begin{align*}
    y\leq l\land \neg(l\land\neg x) & \iff y\leq \neg(l\land\neg x)\\
    & \iff y\land l \land \neg x=0\\
    & \iff y\leq\neg\neg x\\
    & \iff y\leq\neg\neg x\land l
\end{align*}
so by Yoneda, the double negation of $x\in L_{\leq l}$ is given by $\neg\neg x\land l$, i.e. $$\Omega_{\neg\neg}(l)=\{x\leq l\,|\,\neg\neg x\land l=x\}$$
Also for $l\leq l'$, we have the pair of adjoint functors 
\begin{align*}
    \Omega_{\neg\neg}(l\leq l'): & \quad\Omega_{\neg\neg}(l')  \xrightarrow{} \Omega_{\neg\neg}(l)\\
    & \quad x  \mapsto  x\land l
\end{align*}
right adjoint and 
\begin{align*}
    \exists_{l\leq l'}:&\quad\Omega_{\neg\neg}(l)\xrightarrow{}\Omega_{\neg\neg}(l')\\
    &\quad x\mapsto \neg\neg x\land l'
\end{align*}
left adjoint. Now notice that $\exists_{l\leq l'}$ is injective, because for $x,\,x'\in \Omega_{\neg\neg}(l)$, if $\neg\neg x\land l'=\neg\neg x'\land l'$, then $\neg\neg x\land l=\neg\neg x'\land l$ and so $x=x'$. So $\Omega_{\neg\neg}(l\leq l')$ is surjective, and in fact we have $\Omega_{\neg\neg}(l)=\{x\land l\,|\,x\in L,\,\neg\neg x=x\}$.

\begin{proposition}\label{prop_localic}
    Let $L$ be a regular locale, and consider the topos $\mathbf{Sh}(L)$. Then $\textup{Idl}(\Omega_{\neg\neg})(1)=\textup{Idl}^+(\Omega_{\neg\neg}(1))$, where $\textup{Idl}^+$ is the set of ideals $I$ satisfying the following closure property:\\
    for every family $l_i,\,i\in F$ of $\Omega(1)$ such that $\bigvee l_i=1$, for every $x\in \Omega_{\neg\neg}(1)$, if $\neg\neg(x\land l_i)\in I$ for all $i\in F$ then $x\in I$.
\end{proposition}
\begin{proof}
    Consider $P\in \textup{Idl}(\Omega_{\neg\neg})(1)$. $P$ is a $K$-closed sieve on the object $(1,1)$ of $\mathcal{G}(\Omega_{\neg\neg})$. We define $I_P:=\{x\in\Omega_{\neg\neg}\,|\,P(1,x)\ne\varnothing\}$. Clearly it is an ideal: if $x\leq y\in I_P$ then we have a morphism $(1,x)\xrightarrow{}(1,y)$ and hence a map $P(1,y)\xrightarrow{}P(1,x)$ with $P(1,y)$ nonempty, so $P(1,x)$ nonempty; and if $x,\,y\in I_P$ then just use $K_1$-closure of $P$ to get that $x\lor y\in I_P$. It also satisfies the closure property, because if we take $(l_i)_i$ and $x$ satisfying the required property, we want to show that $P(1,x)\ne\varnothing$, but by $K_2$-closure it is enough to prove that $P(l_i,x\land l_i)\ne\varnothing$, and this just follows from $(l_i,x\land l_i)\leq (1,\neg\neg (x\land l_i))$. Thus we defined an order preserving map $P\mapsto I_P$, we still need to check that it is injective and surjective.\\
    \\
    For injectivity, we prove that given such a $P$, we have $P(l,x)\ne\varnothing$ if and only if there is a family $(l_i)_i$ in $\Omega(1)$ such that $l=\bigvee l_i$ and $\neg\neg (x\land l_i)\in I_P$ for all $i$. One direction is easy: if $P(1,\neg\neg(x\land l_i))\ne\varnothing$, then $P(l_i,x\land l_i)\ne\varnothing$ for all $i$, and by $K_2$-closure $P(l,x)\ne\varnothing$. For the other one, we need to use that $L$ is regular. So suppose $P(l,x)\ne\varnothing$. We can chose families $l_i,\,t_i$ such that $l=\bigvee l_i$, $t_i\land l_i=0$ and $t_i\vee l= 1$, and we may assume the $l_i$ to be $\neg\neg$-closed. We want to show that $\neg\neg (x\land l_i)\in I_P$ for all $i$. Now by definition $(1,\neg\neg (x\land l_i))$ is $K$-covered by $(l,\neg\neg (x\land l_i)\land l)=(l,x\land l_i)$ and $(t_i,0)$, but $(t_i,0)$ is $K$-covered by $\varnothing$, so $(1,\neg\neg (x\land l_i))$ is $K$-covered by $(l,x\land l_i)$, and by $K$-closure of $P$ we are reduced to show that $P(l,x\land l_i)\ne\varnothing$, but this just follows from $P(l,x)\ne\varnothing$.\\
    \\
    For surjectivity, for any $I\in \textup{Idl}^+(\Omega_{\neg\neg}(1))$, we can define $P$ by $P(l,x)\ne\varnothing$ if and only if there is a family $(l_i)_i$ in $\Omega(1)$ such that $l=\bigvee l_i$ and $\neg\neg (x\land l_i)\in I$ for all $i$. It is easy to check that this gives an element of $\textup{Idl}(\Omega_{\neg\neg})(1)$. However we need to check that this does not add elements $x\not\in I$ such that $P(1,x)\ne\varnothing$. But this is just a consequence of the closure property. 
\end{proof}

\begin{corollary}
    For a regular locale $L$, $\gamma \mathbf{Sh}(L)=\mathbf{Sh}(\textup{Idl}^+(L_{\neg\neg}))$.
\end{corollary}
\begin{proof}
    Follows from the discussion before the proposition.
\end{proof}

\begin{corollary}\label{gc_top}
    For a compact Hausdorff space $X$, we have $\gamma \mathbf{Sh}(X)=\mathbf{Sh}(\gamma X)$.
\end{corollary}
\begin{proof}
    If the locale is compact, then any ideal satisfies the closure property. We conclude by recalling that the original definition of Gleason's cover for a space $X$ is the Stone space of the Boolean algebra $\Omega_{\neg\neg}(X)$ of regular opens, that is the space whose open sets are ideals of $\Omega_{\neg\neg}(X)$, and that $\mathbf{Sh}(X)$ for $X$ a topological space is by definition $\mathbf{Sh}(O(X))$ where $O(X)$ is the locale of open sets of $X$.
\end{proof}

We can also use the same kind of calculations to prove that in general we cannot have that the canonical morphism $\mathbf{Sh}(\textup{Idl}(\Omega_m),J^{ext})\xrightarrow{}\mathbf{Sh}(\mathcal{C},J)$ is an equivalence if and only $\mathbf{Sh}(\mathcal{C},J)$ is De Morgan. Indeed, if $\mathbf{Sh}(\mathcal{C},J)$ is De Morgan, $\Omega_m=\Omega$. By Proposition~\ref{int_loc_eq} (and Lemma~\ref{triv}), the morphism $\mathbf{Sh}(\mathcal{G}(\textup{Idl}(\Omega)),J^{ext})\xrightarrow{}\mathbf{Sh}(\mathcal{C},J)\cong \mathbf{Sh}(\mathcal{G}(\Omega),J^{ext})$ comes from a unique internal frame morphism $\Omega\xrightarrow{}\textup{Idl}(\Omega)$, which is just the canonical one. Now in the (regular) localic case we can do the same kind of simplification as in Proposition~\ref{prop_localic}. We say that a topological space is \emph{extremally disconnected} if $O(X)$ is a Stone algebra, and \emph{almost discrete} if $O(X)$ is a Boolean algebra.
\begin{proposition}\label{prop_localic2}
    Let $L$ be a regular locale, and consider the topos $\mathbf{Sh}(L)$. Then $\textup{Idl}(\Omega)(1)=\textup{Idl}^{++}(\Omega(1))$, where $\textup{Idl}^{++}$ is the set of ideals $I$ satisfying the following closure property:\\
    for every family $l_i,\,i\in F$ of $\Omega(1)$ such that $\bigvee l_i=1$, for every $x\in \Omega(1)$, if $x\land l_i\in I$ for all $i\in F$ then $x\in I$.
\end{proposition}
\begin{proof}
    Essentially the same proof as Proposition~\ref{prop_localic}.
\end{proof}
Now if $L$ is also compact, clearly $\textup{Idl}^{++}(\Omega(1))=\textup{Idl}(\Omega(1))=\textup{Idl}(L)$. So take any extremally disconnected compact Hausdorff space $X$, the corresponding locale of open sets is in particular compact regular. As discussed before, $\mathbf{Sh}(\textup{Idl}(\Omega_m),J^{ext})\xrightarrow{}\mathbf{Sh}(X)$ is an equivalence if and only if the canonical morphism $\Omega\xrightarrow{}\textup{Idl}(\Omega)$ is an isomorphism, if and only if $\textup{Idl}(O(X))=O(X)$ (i.e. every ideal of $O(X)$ is principal). But this is equivalent to the fact that every open set is compact. Indeed, if this is true then clearly any ideal is closed by arbitrary union, so is principal, and conversely if every ideal of $O(X)$ is principal then considering a situation $x=\bigvee_i x_i$, we can look at the ideal generated by the $x_i$ which is principal, so necessarily it contains $x$, which means that $x$ is smaller than a finite union of the $x_i$, which concludes. Now in the case of compact Hausdorff spaces, since a closed subset of a compact space is compact, this is equivalent to  the property that every open set is closed, i.e. $X$ is almost discrete, which is obviously not shared by all extremally diconnected compact Hausdorff spaces, so we have a counter-example to the fact that $\mathbf{Sh}(\textup{Idl}(\Omega_m),J^{ext})\xrightarrow{}\mathcal{E}$ is an equivalence if and only if $\mathcal{E}$ is De Morgan. It would be interesting to see if the fact that $\textup{Idl}(\Omega)=\Omega$ implies that the topos is in fact Boolean is true in a more general setting, or even more generally if $\textup{Idl}(\Omega_j)=\Omega$ for some (dense) Lawvere-Tierney topology $j$ implies that $\Omega_j=1\sqcup 1$, which is true for $j=\neg\neg$ (it is just a reformulation of the fact that the topos is De Morgan if and only if Gleason's cover is an equivalence), and it is true for $j=\id$ in the above case, because a topos in Boolean if and only if $\Omega=1\sqcup 1$. In fact we can give a more general statement:

\begin{proposition}
    Let $L$ be a regular compact locale. Then in $\mathbf{Sh}(L)$, $\textup{Idl}(\Omega)=\Omega$ if and only if $\Omega=1\sqcup 1$. 
\end{proposition}
\begin{proof}
    Because of Proposition~\ref{prop_localic2}, we just need to show that $\textup{Idl}(\Omega(1))=\Omega(1)$ if and only if the topos is Boolean, i.e. $\textup{Idl}(L)=L$ if and only if $L$ is Boolean. By the same argument as for topological spaces, $\textup{Idl}(L)=L$ if and only if every element of $L$ is compact. We claim that in a compact regular locale, compact elements are exactly complemented elements (``clopen sets are exactly compact open sets"), which finishes the proof. Let $x\in L$ be a compact element. By regularity, we can write $x$ as the join of a family $x_i$, $i\in I$, such that for each $i$ there is $t_i$ with $t_i\land x_i=0$ and $t_i\lor x=1$. Now by compactness $I$ can be taken to be finite. Now consider $t:=\bigwedge_i t_i$. We have $t\land x=\bigvee_i(t\land x_i)\leq\bigvee_i(t_i\land x_i)=0$. Also by the distributivity law, $1=\bigwedge_i(x\lor t_i)=\bigvee_i(x\land t_i)\lor x\lor t= x\lor t$. So $t$ is indeed the complement of $x$. Conversely, if $x$ is complemented, for a situation $x=\bigvee_{i\in I} x_i$, we have $\bigvee_i x_i\lor \neg x=1$, so we can extract a finite $J\subseteq I$ such that $\bigvee_{i\in J} x_i\lor \neg x=1$, so $\bigvee_{i\in J} x_i$ is the complement of $\neg x$, and by uniqueness of the complement $x=\bigvee_{i\in J} x_i$. 
\end{proof}

We expect a more general statement to hold, which would take the form of a fibred Stone duality. This is the subject of a future work.

\subsubsection{Gleason's cover of a presheaf topos}\label{gc_psh}
Now we want to answer the following question: when is Gleason's cover of a presheaf topos again a presheaf topos? This question is interesting, because by Proposition~\ref{Ore}, $\mathbf{Psh}(\mathcal{C}^{op})$ is De Morgan if and only if $\mathcal{C}$ has amalgamation, so this could be a way of enforcing amalgamation in some category: starting from a category $\mathcal{C}$, we can consider $\gamma \mathbf{Psh}(\mathcal{C}^{op})$, and if this is equivalent to $\mathbf{Psh}(\mathcal{C}')$ for some $\mathcal{C}'$, $(\mathcal{C}'^{op})$ could be some kind of `completion for amalgamation' of $\mathcal{C}$.\\
\\
We have the following classical proposition: 
\begin{proposition}[\cite{johnstone2002sketches} Lemma~C2.2.20]
    A topos $\mathcal{E}$ is a presheaf topos if and only if every object of $\mathcal{E}$ is $J^{can}$-covered by a set of irreducible objects, i.e. objects $E$ such that the only $J^{can}$-covering sieve is the maximal sieve, or in other words every epimorphic family on $E$ contains a split epi.
\end{proposition}
This is just a consequence of the \emph{Comparison Lemma}, see for example Corollary~4.3 of the appendix of \cite{maclane2012sheaves}; however, earlier formulations of this Theorem~exist, for instance Theorem~13.1 of \cite{bunge1966categories}, where irreducible objects were called ``atoms", and exercise 7.6(d) of \cite{grothendieck2006sga}, where it is attributed to Roos. Note that the notion of irreducible object as a topos-theoretic invariant was introduced by Caramello in \cite{caramello2012syntactic}, Definition 2.3(c).\\
\\
More generally, for a site $(\mathcal{C},J)$, we say that the site is \emph{rigid} if every object is $J$-covered by $J$-irreducible objects, where an object is $J$-irreducible if the only $J$-covering sieve on it is the maximal one. Note that if a topos has a rigid site of definition, then it is equivalent to a presheaf topos (because if an object $c$ of $\mathcal{C}$ is $J$-irreducible, then $l_J(c)$ is irreducible, essentially by Corollary~III.7.7 of \cite{maclane2012sheaves} and Lemma~1.3.8 of \cite{makkai2006first}), but all the sites of definition of a presheaf topos do not have to be rigid. However, it is the case for every subcanonical site $(\mathcal{C},J)$ with $\mathcal{C}$ Cauchy-complete, by \cite{caramello2018theories} Theorem~6.1.17 and Remark~6.1.18. We start by studying when $K_1$ is rigid. Recall that an \emph{atom} of a Boolean algebra is an element $x$ such that $x\ne 0$ and if $x'\leq x$ then $x'=0$ or $x'=x$.
\begin{Lemma}\label{rigid}
    Let $\mathcal{E}:=\mathbf{Sh}(\mathcal{C},J)$. An object $(c,x)$ of $\mathcal{G}(\Omega_{\neg\neg})$ is $K_1$-irreducible if and only if $x$ is an atom. Thus $K_1$ is rigid if and only if $\Omega_{\neg\neg}(c)$ is finite for all $c$
\end{Lemma}
\begin{proof}
    In a Boolean algebra, an atom is equivalently an element $x\ne 0$ such that $x=\bigOR_i x_i$ implies $x=x_i$ for some $i$, so the first claim is straightforward. It follows that $K_1$ is rigid if and only if for all $c$ we have $\Omega_{\neg\neg}(c)$ atomic and every element is the union of a finite number of atoms, if and only if $\Omega_{\neg\neg}(c)$ is finite.
\end{proof}
This gives a sufficient condition for the Gleason cover of a presheaf topos to be a presheaf topos, and it was already observed in \cite{johnstone1980gleason}, section 4. This is not enough to give the converse, since $K_1$ is not subcanonical. Indeed, every object $(c,0)$ is covered by the empty sieve, so all the $l((c,0))$ are isomorphic to the initial object of the topos. However, we can get it with a little more work, but we need to study the irreducible objects of $\gamma\mathcal{E}$. 

\begin{proposition}
    Let $\mathcal{E}:=\mathbf{Sh}(\mathcal{C},J)$ be any topos, with $(\mathcal{C},J)$ a subcanonical site and $\mathcal{C}$ Cauchy-complete. The irreducible objects of $\mathbf{Sh}(\mathcal{G}(\Omega_{\neg\neg}),K)$ are all of the form $l((c,x))$ with $x\ne 0$.
\end{proposition}
\begin{proof}
    Let $E$ an irreducible. $E$ is covered by objects of the form $l((c_i,x_i))$ (for which we can suppose $x_i\ne 0$), and since $E$ is irreducible, one of these morphisms is a split epi. So 
    we have a retraction 
    $$E\xrightarrow{i}l((c_i,x_i))\xrightarrow{e}E$$
    for some $i$, i.e. we have $e\circ i=\id$. Since $x_i\ne 0$ and $(\mathcal{C},J)$ is subcanonical, it is easy to see that $\yo((c_i,x_i))$ is already a $K$-sheaf, so $l((c_i,x_i))=\yo((c_i,x_i))$. But noticing that $\mathcal{C}$ Cauchy-complete implies $\mathcal{G}(\Omega_{\neg\neg})$ Cauchy-complete, this is enough to make Remark~6.1.8 of \cite{caramello2018theories} work, so $E$ is isomorphic to an object coming from the site.
\end{proof}

\begin{proposition}
    Let $\mathcal{E}:=\mathbf{Psh}(\mathcal{C})$ be a presheaf topos, with $\mathcal{C}$ Cauchy-complete. The irreducible objects of $\mathbf{Sh}(\mathcal{G}(\Omega_{\neg\neg}),K)$ are exactly the objects $l((c,x))$ with $x$ an atom.
\end{proposition}
\begin{proof}
    By the previous proposition, we just have to show that $l((c,x))$ is irreducible if and only if $x$ is an atom. One direction is clear. Now suppose $l((c,x))$ is irreducible, and $x=\bigOR_i x_i$. Note that $x\ne 0$ because the initial object is not irreducible. There is some $i$ such that $l(\id):l((c,x_i))\xrightarrow{}l((c,x))$ is a split epimorphism, so since it is also monic, it is an isomorphism. As above, $x$ and $x_i$ are not zero so $l((c,x))=\yo((c,x))$ and similarly for $x_i$, and we conclude by the Yoneda Lemma~that $x=x_i$.
\end{proof}

We are ready to prove the full characterization.

\begin{theorem}\label{thm:psh_gc}
    The Gleason cover of a presheaf topos $\mathbf{Psh}(\mathcal{C})$ with $\mathcal{C}$ Cauchy-complete is a presheaf topos if and only if $\Omega_{\neg\neg}(c)$ is finite for every object $c$ of $\mathcal{C}$.
\end{theorem}
\begin{proof}
    For a presheaf topos, the topology $K$ reduces to $K_1$. So by virtue of Lemma~\ref{rigid}, we just need to show that $\gamma\mathbf{Psh}(\mathcal{C})$ is a presheaf topos if and only if $K_1$ is rigid. Only one direction needs a proof. So suppose $\gamma\mathbf{Psh}(\mathcal{C})$ is a presheaf topos. Every object of $\gamma\mathbf{Psh}(\mathcal{C})$ is covered by a family of irreducible objects which, by the last proposition, are of the form $l((d,y))$ with $y$ an atom. So if we take an object $l((c,x))$ with $x\ne 0$, it is covered by such a family $(l((d_i,y_i)))_{i\in I}$ and we can suppose that the morphisms all come from the site since $y_i\ne 0$ for all $i$ and so $l((d_i,y_i))=\yo((d_i,y_i))$ and we can apply the Yoneda lemma. By Corollary~III.7.7 of \cite{maclane2012sheaves}, $(c,x)$ is $K_1$-covered by $((d_i,y_i))_{i\in I}$, and we conclude by the definition of $K_1$.
\end{proof}

\noindent In particular, this is true if $\mathcal{C}$ is finite. Now we introduce the following construction:
\begin{definition}
    Let $\mathcal{C}$ be a category. Let $c$ be an object of $\mathcal{C}$. A sieve $S$ on $c$ is \emph{$\neg\neg$-closed} if for every $f:d\to c$ 
\[(\forall g:e\to d,\,\exists h:e'\to e\text{ such that } f\circ g\circ h\in S) \text{ implies } f\in S\]
    We define $\At_{\neg\neg}(\mathcal{C})$ to be the category whose objects are pairs $(c,S)$ with $c$ an object of $\mathcal{C}$ and $S$ a minimal (for inclusion) nonempty $\neg\neg$-closed sieve on $c$, and whose morphisms are of the form $f:(d,S')\to (c,S)$ with $f:d\to c$ a morphism of $\mathcal{C}$ and $S=\{f\circ g\,|\,g\in S'\}=:\exists_f(S')$. We define $\pi_\mathcal{C}:\At_{\neg\neg}(\mathcal{C})\to\mathcal{C}$ to be the projection.
\end{definition}

Note that in the above definition, a $\neg\neg$-closed sieve on $c$ is $\neg\neg$-closed as an element of $\Omega(c)$ with $\Omega$ the subobject classifier of $\Psh(\mathcal{C})$, so a minimal $\neg\neg$-closed sieve on $c$ is exactly an atom of $\Omega_{\neg\neg}(c)$. Moreover, since $S$ is minimal, $S=\exists_f(S')$ is equivalent to $\exists_f(S')\subseteq S$ which in turn is equivalent to $S'\subseteq f^*(S)$, (which recall is the sieve $\{g:\dom(g)\to d\,|\,f\circ g\in S\}$), so $\At_{\neg\neg}(\mathcal{C})$ is just the full subcategory of $\mathcal{G}(\Omega_{\neg\neg})$ consisting of objects $(c,x)$ with $x$ an atom of $\Omega_{\neg\neg}(c)$. By an immediate generalization of~\cite[Theorem 4.2]{johnstone1980gleason}, under the condition of Theorem~\ref{thm:psh_gc}, the geometric morphism induced by $\pi_\mathcal{C}$ coincides with Gleason's cover. So $\At_{\neg\neg}(\mathcal{C})$ is the `completion of $\mathcal{C}$ under amalgamation'. Note that, in some cases, notably when $\mathcal{C}$ is finite, we have a 1-1 correspondence between atoms of $\Omega_{\neg\neg}(c)$ and minimal elements of $\Omega(c)$, which makes it even easier to compute (see for example by~\cite[Lemma 4.1]{johnstone1980gleason}). However, this construction is still interesting if we assume weaker conditions. 

\begin{theorem}
    Let $\Psh(\mathcal{C})$ be a presheaf topos such that $\Omega_{\neg\neg}(c)$ contains an atom for every $c$ in $\mathcal{C}$. Then $\At_{\neg\neg}(\mathcal{C})$ satisfies the right Ore condition, and the geometric morphism $C_{\pi_\mathcal{C}}:\Psh(\At_{\neg\neg}(\mathcal{C}))\to\Psh(\mathcal{C})$ induced by $\pi_\mathcal{C}$ is surjective. Moreover, if $\mathcal{C}$ already satisfies the right Ore condition, then $C_{\pi_\mathcal{C}}$ is an equivalence.
\end{theorem}
\begin{proof}
    $\At_{\neg\neg}(\mathcal{C})$ satisfies the right Ore condition because the proof given in~\cite[Remark 4.3]{johnstone1980gleason} still works. We give the idea, to concretely illustrate why this is the case and where the hypothesis on $\mathcal{C}$ is needed. 
    Consider a diagram 
    \[(c,x)\xrightarrow{f}(e,z)\xleftarrow{g}(d,y)\]
    in $\At_{\neg\neg}(\mathcal{C})$. By definition, $\exists_f(x)=\exists_g(y)(=z)$, so for any $p\in x$, there is $p'\in y$ such that $f\circ p=g\circ p'$. Now since $p\in x$ and $p'\in y$, for any $t\in\Omega_{\neg\neg}(\dom(p))$, which exists by hypothesis, we have $\exists_p(t)\subseteq x$ and $\exists_{p'}(t)\subseteq y$, i.e.~we have a commutative diagram:
    \[\begin{tikzcd}
        (\dom(p),t)\arrow[r,"p"]\arrow[d,"p'"'] & (c,x)\arrow[d,"g"]\\
        (d,y)\arrow[r,"f"'] & (e,z)
    \end{tikzcd}
    \]
    
    $C_{\pi_\mathcal{C}}$ is surjective by~\cite[Proposition 7.1]{caramello2019denseness} (this also uses the hypothesis). If $\mathcal{C}$ satisfies the right Ore condition, $\Psh(\mathcal{C})$ is De Morgan, so $\Omega_{\neg\neg}\cong 1\sqcup1$ which, since we are in a presheaf topos, is just the pointwise two elements Boolean algebra, which pointwise admits exactly one atom (the maximal element). So $\pi_\mathcal{C}:\At_{\neg\neg}(\mathcal{C})\to \mathcal{C} $ is an isomorphism.
\end{proof}

\noindent Thus, this construction satisfies most of the nice properties of Gleason's cover (and coincide with it if the condition of Theorem~\ref{thm:psh_gc} is satisfied), but has the advantage of always giving a presheaf topos, hence always enforcing amalgamation. 
We end this section by illustrating the above construction on the basic example of a category \[\mathcal{C}:=c\xrightarrow{f} e \xleftarrow{g} d\]
This category does not satisfy the right Ore condition, i.e.~its opposite does not satisfy amalgamation. Following the above explicit description, $\At_{\neg\neg}(\mathcal{C})$ is the category:
\begin{center}
    \begin{tikzcd}
        (e,\{f\}) & (e,\{g\})\\
        (c,\{\id\})\arrow[u, "f"] & (d,\{\id\})\arrow[u, "g"']
    \end{tikzcd}
\end{center}
which indeed satisfies the right Ore condition. 
It is interesting to point out that this construction does not freely add amalgams, but rather dislocates the diagrams when an amalgamation is missing, which is dual in some sense.

\subsection{Enforcing De Morgan's law pointwise}\label{dml_pt}
The internal locale associated to Gleason's cover can be though of as the `global' ideal completion of $\Omega_{\neg\neg}$. However, one can wonder if the `local' ideal completion of some internal locale, in particular of $\Omega_m$ where $m$ is the De Morgan topology of \cite{caramello2009morgan}, can be of some interest. We first study the question in full generality. 

\begin{proposition}
    Let $L$ be an internal locale in $\mathbf{Psh}(\mathcal{C})$, $\mathcal{C}$ cartesian. Then $I_L: c\mapsto \textup{Idl}(L(c))$ is an internal locale of $\mathbf{Psh}(\mathcal{C})$.
\end{proposition}
\begin{proof}
    It is well-known that the ideals of a frame form a frame. The transition function for $f:d\xrightarrow{}c$ is given by $I_L(f): u \mapsto \downarrow (L(f)[u])$, i.e. the set of elements of $L(d)$ that are smaller than an element in the image of $u$. Similarly, the left adjoint is given by $\exists_f^{I_L}: u \mapsto \downarrow (\exists_f^{L}[u])$. By the fact that $L(f)$ and $\exists_f^L$ preserve binary joins, it is easily checked that these operations indeed produce ideals of the corresponding frames. We  need to check that $I_L(f)$ preserves arbitrary joins and meets. Recall that the join $\bigvee_i u_i$ of a family $u_i$ of ideals is the ideal whose underlying set is the set of elements which are smaller than a finite join of elements in $\bigcup_i u_i$, so $I_L(f)(\bigvee_i I_i)=\bigvee_i I_L(f)(u_i)$ just by the fact that $L(f)$ preserves joins. The meet is just given by set-theoretic intersection, so the fact that $I_L(f)$ preserve these is again given by the corresponding property for $L(f)$ (if $x\leq L(f)(y_i)$ with $y_i\in I_i$, then $x\leq \bigwedge_i L(f)(y_i)=L(f)(\bigwedge_i y_i)$ with $\bigwedge_i y_i\in\bigwedge_i u_i$). We also check that the left adjoint is indeed given by $I \mapsto \downarrow (\exists_f^{L}[I])$, by checking that this is a left adjoint to $I_L(f)$, i.e. we have to check that $u\subseteq \downarrow (L(f)[v])$ if and only if $\downarrow (\exists_f[u])\subseteq v$. Right to left direction: if $x\in u$, then $\exists_f (x)\in v$, and $x\leq L(f)(\exists_f)(x)$. Left to right direction: if $y\leq \exists_f(x)$ for some $x\in u$, knowing that $x\leq L(f)(z)$ for $z\in v$, we get $y\leq \exists_f (L(f)(z))\leq z$ so $y\in v$.\\
    \\
    We still need to check Beck-Chevalley and Frobenius. In fact Beck-Chevalley immediately follows from the corresponding property for $L$. For Frobenius, we have to check that $\downarrow (\exists_f[(\downarrow (L(f)[u]))\land v])= u \land (\downarrow (\exists_f[v]))$. Right to left inclusion: if $x\in u$ and $x\leq\exists_f(y)$ with $y\in v$, then $x=x\land\exists_f(y)=\exists_f(L(f)(x)\land y)$ by Frobenius. Left to right: if $x\leq \exists_f(y)$ with $y\in v$ and $y\leq L(f)(z)$ with $z\in u$, then $x\leq \exists_f(L(f)(z))\leq z$. 
\end{proof}

We will use the same notation $\exists_f$ for the right adjoint in $L$ and $I_L$ when there is no ambiguity. Notice that we only get an internal locale of $\mathbf{Psh}(\mathcal{C})$, so we still need to apply an additional operation to get an internal locale of $\mathbf{Sh}(\mathcal{C},J)$ (see the end of this appendix). However, in some cases there is no need to do so. For a site $(\mathcal{C},J)$, we say that $J$ is \emph{finitary} if every $J$-covering sieve is generated by a finite family of morphisms. For example, the coherent topology on a coherent category is finitary. Also the trivial topology is (trivially) finitary, since any covering sieve is generated by the identity morphism.

\begin{proposition}
    With the same notations, if $L$ is a $J$-sheaf and $J$ is finitary then $I_L$ is also an internal locale of $\mathbf{Sh}(\mathcal{C},J)$.
\end{proposition}
\begin{proof}
    We just have to check that $I_L$ is a $J$-sheaf. So take a (finite) covering family $f_i:c_i\xrightarrow{}c$ and a matching family $u_i\in \textup{Idl}(L(c_i))$. We claim that $u:=\bigvee_i\exists_{f_i}(u_i)$ is an amalgamation of this family. Indeed, $I_L(f_i)(u)\geq I_L(f_i)(\exists_{f_i}(u_i))\geq u_i$, and we are left to show that $I_L(f_i)(u)\leq u_i$, equivalently that $I_L(f_i)(\exists_{f_j}(u_j))\leq u_i$ for all $i,\,j$. Now consider the following pullback square:
    \begin{center}
    \begin{tikzcd}
    c_i\times_c c_j\arrow[r,"p_i"]\arrow[d,"p_j"'] & \arrow[d,"f_i"] c_i\\
    c_j\arrow[r,"f_j"'] & c
    \end{tikzcd}
    \end{center}
    Using Beck-Chevalley, we get 
    $$I_L(f_i)(\exists_{f_j}(u_j))=\exists_{p_i}(I_L(p_j)(u_j))=\exists_{p_i}(I_L(p_i)(u_i))\leq u_i$$
    the last equality coming from the fact that $(u_i)_i$ is a matching family. This proves the claim.\\
    \\
    We still need to check uniqueness. We claim that for all $u\in \textup{Idl}(L(c))$, for every (finite) covering family $f_i:c_i\xrightarrow{}c$, $i\in I$, we have $u=\bigvee_i \exists_{f_i}(I_L(f_i)(u))$. Clearly this gives uniqueness, so proving the claim finishes the proof. Now we have 
    $$\bigvee_i \exists_{f_i}(I_L(f_i)(u))=\bigvee_i \exists_{f_i}(I_L(f_i)(u)\land 1)=\bigvee_i (\exists_{f_i}(1)\land u)=u\land \bigvee_i \exists_{f_i}(1)$$
    by Frobenius, so we just have to prove that $\bigvee_i \exists_{f_i}^{I_L}(1)=1$. Unwrapping the definitions, we have to show that there are $x_i\in L(c_i)$ such that $1\leq \bigvee_{i\in F}\exists_{f_i}(x_i)$ in $L(c)$, for a finite $F$. Equivalently, we have to show that $1\leq \bigvee_{i\in I}\exists_{f_i}(1)$, since $I$ is finite. Now starting from $\bigvee_{i\in I}\exists_{f_i}(1)\leq \bigvee_{i\in I}\exists_{f_i}(1)$, we get that $1\leq L(f_i)(\bigvee_{i\in I}\exists_{f_i}(1))$ for all $i$, so since $L$ is a sheaf and $f_i$ is covering, we get what we wanted.
\end{proof}

Now one can wonder if this local completion is in fact a fibred ideal completion with respect to some Grothendieck topology on $\mathcal{G}(L)$, in the sense of section \ref{sec:int_loc}. If we still assume that $J$ is finitary, then the answer is positive. 

\begin{definition}
    Let $L$ be an internal locale of some topos $\mathbf{Sh}(\mathcal{C},J)$, with $\mathcal{C}$ cartesian. The finitary existential topology $J^{finext}$ on $\mathcal{G}(L)$ is the topology generated by the basis defined by finite families $(c_i,x_i)\xrightarrow{f_i}(c,x)$ such that $\bigvee_i \exists_{f_i}(x_i)=x$.
\end{definition}

In fact, it is not obvious that this defines a topology on $\mathcal{G}(L)$, specifically the stability under pullbacks. But since the existential topology (which is just the infinitary version) is stable under pullbacks because $L$ is an internal locale, and since the pullback (in the strong sense) of a finite family is finite, then $J^{finext}$ is indeed a Grothendieck topology on $\mathcal{G}(L)$. Now the question is whether the projection $\mathcal{G}(L)\xrightarrow{}C$ generates a localic morphism $\mathbf{Sh}(\mathcal{G}(L),J^{finext})\xrightarrow{}\mathbf{Sh}(\mathcal{C},J)$, i.e. if $J^{finext}$ contains the Giraud topology. This is where we need $J$ to be finitary.

\begin{proposition}
    If $J$ is finitary then the topology $J^{finext}$ on $\mathcal{G}(L)$ for $L$ internal locale contains the Giraud topology.
\end{proposition}
\begin{proof}
    In fact we can see $J^{finext}$ as the existential topology on the existential fibred site $L$ where we put the coherent topology on each fibre (see Theorem~5.1 of \cite{caramello2022fibred}). The relative Beck-Chevalley and Frobenius conditions will still be valid in the cartesian case, as it is clear from the definition. Then we can apply  Proposition~5.4 of \cite{caramello2022fibred}. So we just need to show that this fibred site is $J$-reflecting (definition 5.1(d) of \cite{caramello2022fibred}). Consider a (finite) $J$-covering family $S$ on $c$ and a family $x_i\leq x$ in $L(c)$, $i\in I$. Suppose also that for every $f\in S$, there is a finite $I_f$ with $\bigvee_{i\in I_f}L(f)(x_i)=L(f)(x)$. We need to show that there is a finite $F\subseteq I$ such that $\bigvee_{i\in F}x_i=x$. Define $F:=\bigcup_{f\in S}I_f$. For every $f\in S$, we have $L(f)(\bigvee_{i\in F}x_i)\geq L(f)(\bigvee_{i\in I_f}x_i)=L(f)(x)$, and this is an equality because $x_i\leq x$ for all $i$. So since $L$ is a sheaf and $S$ is covering, we have indeed $x=\bigvee_{i\in F}x_i$.
\end{proof}

In fact we can give yet another description of $J^{finext}$. Indeed, it is not hard to see that this topology is the one generated by all the families containing a $(d,x)\xrightarrow{f}(c,\exists_f(x))$ for $f:d\xrightarrow{}c$ in $\mathcal{C}$ and the finite families $(c,x_i)\xrightarrow{id}(c,x)$ with $x=\bigvee_i x_i$. Let us call them covers of type 1 and 2 respectively. Now as we mentioned before the proposition, the morphism $\mathbf{Sh}(\mathcal{G}(L),J^{finext})\xrightarrow{}\mathbf{Sh}(\mathcal{C},J)$ generated by the projection is localic, so it corresponds, by Proposition~\ref{int_loc_eq}, to an internal locale $\Tilde{L}$, and for any other locale $L'$ we have (using Theorem~3.3 and Proposition~4.2 of \cite{caramello2022fibred}):
$$Loc(L',\Tilde{L})=Fib^{cart,cov}((\mathcal{G}(L),J^{finext}),(\mathcal{G}(L'),J^{ext}))$$
Following \cite{wrigley2023some}, we see that an \emph{internal locale morphism} $\alpha:L'\xrightarrow{}\Tilde{L}$ is a natural transformation $\alpha^{-1}:\Tilde{L}\xrightarrow{}L'$ which is pointwise a frame morphism, and which is also \emph{conatural}, in the sense that for any morphism $f:d\xrightarrow{}c$ in $\mathcal{C}$, the following square commutes:
\begin{center}
    \begin{tikzcd}
    \Tilde{L}(d)\arrow[r,"\exists_f"]\arrow[d,"\alpha_d"'] & \arrow[d,"\alpha_c"] \Tilde{L}(c)\\
    L'(d)\arrow[r,"\exists_f"'] & L'(c)
    \end{tikzcd}
    \end{center}

Now like in the proof of Corollary~3.5 of \cite{caramello2022fibred}, we see that a cartesian cover-preserving fibration morphism $F:(\mathcal{G}(L),J^{finext})\xrightarrow{}(\mathcal{G}(L'),J^{ext})$ corresponds to a natural transformation $\alpha:L\xrightarrow{}L'$ which preserves finite meets (cartesian) and joins (preservation of covers of type 2) pointwise (fibration morphism) and is conatural (preservation of covers of type 1). Note that naturality comes from the preservation of cartesian arrows, which is included in the definition of a fibration morphism. Indeed, we can prove that cartesian arrows in $\mathcal{G}(L)$ are exactly the arrows $(d,L(f)(x))\xrightarrow{f}(c,x)$. Let $(d,x')\xrightarrow{f}(c,x)$ be a cartesian arrow. We have $x'\leq L(f)(x)$, we want to show $x'\geq L(f)(x)$, but this is just a consequence of the fact that we can lift the identity arrow $d\xrightarrow{}d$ to make the following triangle commute:
\begin{center}
    \begin{tikzcd}
    (d,L(f)(x))\arrow[dr,"f"]\arrow[d,"\id"'] &  \\
    (d,x')\arrow[r,"f"'] & (c,x)
    \end{tikzcd}
    \end{center}
The converse is trivial. Since we have the same for cartesian arrows in $\mathcal{G}(L')$, this shows that $F$ preserves cartesian arrows if and only if $\alpha$ is natural. \\
\\
Now we recall from, say, \cite{caramello2011topos} Theorem~6.2, that for frames $L$ and $L'$ (in $Sets$), the morphism sending an element $l$ of $L$ to the principal ideal generated by $l$ induces by precomposition a natural isomorphism between morphisms $L\xrightarrow{}L'$ that preserve finite meets and joins and frame morphisms $\textup{Idl}(L)\xrightarrow{}L'$ (in fact $L$ does not have to be a frame). It is straightforward to extend this result to internal locales and the pointwise ideal completion, and we get that internal frame morphisms between $I_L$ and $L'$ correspond to natural and conatural transformations from $L$ to $L'$ which preserve finite meets and joins pointwise. Putting everything together, by the Yoneda lemma, we have proven that $I_L\cong \Tilde{L}$, i.e. the locale corresponding to the local ideal completion of $L$ is exactly the completion of $L$ with respect to $J^{finext}$. In other words:

\begin{theorem}\label{loc_ideal}
    For $L$ internal locale in $\mathbf{Sh}(\mathcal{C},J)$, $\mathcal{C}$ cartesian, $J$ finitary, we have $$\mathbf{Sh}(\mathcal{G}(L),J^{finext})\cong \mathbf{Sh}(\mathcal{G}(I_L),J^{ext})$$ as relative toposes over $\mathbf{Sh}(\mathcal{C},J)$.
\end{theorem}

We can also study when the morphism $\mathbf{Sh}(\mathcal{G}(I_L),J^{ext})\xrightarrow{}\mathbf{Sh}(\mathcal{C},J)$ is surjective. First, we need a lemma. Recall that the Yoneda embedding $\mathcal{E}\xrightarrow{}\mathbf{Sh}(\mathcal{E},J^{can})$ is an equivalence, so every internal locale $L:\mathcal{C}^{op}\xrightarrow{}\Set$ in $\mathcal{E}=\mathbf{Sh}(\mathcal{C},J)$ can be seen as an internal locale $L:\mathcal{E}^{op}\xrightarrow{}\Set$.

\begin{Lemma}
    For an internal locale $L:\mathcal{E}^{op}\xrightarrow{}\Set$, the following are equivalent:
    \begin{enumerate}[(i)]
        \item The canonical morphism $f:\mathbf{Sh}(\mathcal{G}(L),J^{ext})\xrightarrow{}\mathcal{E}$ is surjective;
        \item $f^*$ reflects epimorphisms;
        \item for any $\alpha:E'\xrightarrow{}E$ in $\mathcal{E}$, if $\exists_\alpha(1)=1$ then $\alpha$ is an epimorhism.
    \end{enumerate}
\end{Lemma}
\begin{proof}
    2 and 3 are clearly equivalent since $\exists_\alpha(1)=Im(f^*(\alpha))$. Now suppose 1. $f$ is generated by the morphism of sites $F:E\mapsto (E,1)$. Take $\alpha:E'\xrightarrow{}E$ such that $\exists_\alpha(1)=1$. So $(E',1)\xrightarrow{\alpha}(E,\exists_\alpha(1))=F(E'\xrightarrow{\alpha}E)$ is $J^{ext}$-covering, and by \cite{caramello2019denseness} Theorem~6.3(i), $\alpha$ is $J^{can}$-covering, i.e. an epimorphism. Conversely, if $f^*$ reflects epimorphisms, we show that it is faithful. Indeed, take $\alpha,\,\alpha':E'\xrightarrow{}E$ such that $f^*(\alpha)=f^*(\alpha')$. Let $E''\xrightarrow{e}E'$ be the equaliser of $\alpha$ and $\alpha'$. $f^*$ preserves finite limits, so $f^*(e)$ is the equaliser of equal morhisms, i.e. an isomorphism. In particular it is an epimorphism, so $e$ is an epimorphism and a monomorphism so it is an isomorphism, and $\alpha=\alpha'$.
\end{proof}

\begin{proposition}\label{surj_ideal}
    $L$ internal locale in a topos $\mathcal{E}=\mathbf{Sh}(\mathcal{C},J)$, $\mathcal{C}$ cartesian, $J$ finitary. Then $\mathbf{Sh}(\mathcal{G}(I_L),J^{ext})\xrightarrow{}\mathcal{E}$ is surjective if and only if $\mathbf{Sh}(\mathcal{G}(L),J^{ext})\xrightarrow{}\mathcal{E}$ is.
\end{proposition}
\begin{proof}
    Just apply the previous Proposition~to $I_L$, noticing that $\exists_\alpha^{L}(1)=1$ if and only if $\exists_\alpha^{I_L}(1)=1$.
\end{proof}

The Proposition~shows that $\mathbf{Sh}(\mathcal{G}(I_{\Omega_m}),J^{ext})\xrightarrow{}\mathbf{Sh}(\mathcal{C},J)$ it is not surjective unless the topos is De Morgan, because $\mathbf{Sh}(\mathcal{G}(\Omega_m),J^{ext})$ is just the DeMorganization of $\mathcal{E}$ which is an inclusion. We have the same problem if we consider $I_{\Omega_{\neg\neg}}$.\\
\\
At this point, one can ask several questions. First, what can we do if $J$ is not finitary? In the general case, we can still get an internal locale by the following operation: for an internal locale $L$ of $\mathbf{Psh}(\mathcal{C})$, the pullback of $\mathbf{Sh}_{\mathbf{Psh}(\mathcal{C})}(I_L)$ along $\mathbf{Sh}(\mathcal{C},J)\hookrightarrow{}\mathbf{Psh}(\mathcal{C})$ is localic, so it is of the form $\mathbf{Sh}_{\mathbf{Sh}(\mathcal{C},J)}(\Tilde{I}_L)$ for some internal locale $\Tilde{I}_L$ of $\mathbf{Sh}(\mathcal{C},J)$. On the other hand, we can also add the Giraud topology to $J^{finext}$. Do we have a result similar to Theorem~\ref{loc_ideal} in this case? The answer is positive, and in fact, we can state a more general fact: take a site $(\mathcal{C},J)$, a fibred preorder $\pcat$, and any topology $K$ on $\mathcal{G}(\pcat)$. Let $G_J$ be the Giraud topology over $J$. The fibred ideal completion of $\pcat$, $L_\pcat^K$, is an internal locale of $\mathbf{Psh}(\mathcal{C})$. By the proof of Proposition~3.4 of \cite{giraud1972classifying}, we have that $\mathbf{Sh}(\Tilde{L}_\pcat^K)\cong \mathbf{Sh}(\mathcal{G}(\pcat), K\lor G_J)$. In particular, this is true in the case where $\pcat$ is an internal distributive lattice and $K$ is the finite covering topology on each fiber, as noticed in Corollary~C2.4.3 of \cite{johnstone2002sketches}, which is already interesting because this is exactly the setting of Gleason's cover. In particular, if $J$ is dense, then $\Omega_{\neg\neg}^{\mathbf{Psh}(\mathcal{C})}=\Omega_{\neg\neg}^{\mathbf{Sh}(\mathcal{C},J)}$ (as elements of $\mathbf{Psh}(\mathcal{C})$), so we have the following pullback diagram of toposes:
\begin{center}
    \begin{tikzcd}
        \gamma \mathbf{Sh}(\mathcal{C},J) \arrow[r]\arrow[d] & \gamma \mathbf{Psh}(\mathcal{C})\arrow[d]\\
        \mathbf{Sh}(\mathcal{C},J)\arrow[r, hook] & \mathbf{Psh}(\mathcal{C})
    \end{tikzcd}
\end{center}

This could be relevant for the study of Gleason's cover, which will be done in a future work. \\
\\
Also we could persue several generalizations, for example when $\mathcal{C}$ is not cartesian, and when $L$ is not an internal locale but just an internal distributive lattice for example. We can also try to generalize this to general $J$-ideals like in \cite{caramello2011topos}, instead of only looking at the coherent topology on a distributive lattice. These generalizations seem interesting and will be done in the future work about fibred Stone duality.

\section{Conclusion and future work}
This paper, despite having some new results, is mainly to be thought of as a methodological clarification on De Morgan's law in toposes. Its goal is to show how to reformulate known results and constructions about De Morgan's law in the framework of relative topos theory, and to hint at possible applications of this point of view. The proofs of section \ref{sec:enforce} clarify the proofs of \cite{johnstone1980gleason} which can seem rather mysterious and heavily rely on known results about distributive lattices. In some cases, the proofs are drastically simplified, for instance the proof of Corollary~\ref{gc_surj}, which is a direct consequence of general results of relative topos theory, and can be compared to the proof of Corollary~1.4 (and thus Lemma~1.3) of \cite{johnstone1980gleason}. This point of view also allows more modularity in the process of enforcing De Morgan's law, as showed by the new construction of section \ref{dml_pt}. We intend to study in more generality the different ways of enforcing De Morgan's law via (generalized) ideal completion, which passes through a study of Stone duality in the relative setting, as already mentioned. \\
\\
The other important protagonist of this paper is the amalgamation property. The link between De Morgan's law and the amalgamation property for the category of models of some theory was already well-known for theories of presheaf type, and some independent results for coherent theories were also available, outside the realm of topos theory. One of the achievements of this paper has been to establish links between all theses scattered results via topos theory. In the same vein, one of our future objectives is the application of our methods for enforcing De Morgan's law to the construction of techniques for enforcing the amalgamation property on a category, as hinted at in section \ref{gc_psh}.

\bibliographystyle{alpha}
\bibliography{b}

\vspace{1cm}
\textbf{Acknowledgements:} Olivia Caramello has benefited for this work of the support of the Université Paris-Saclay in the framework on the Jean D’Alembert 2024 Programme.

\vspace{1cm}
\textsc{Olivia Caramello\\
\\
Dipartimento di Scienza e Alta Tecnologia, Università degli Studi dell’Insubria, via Valleggio 11, 22100 Como, Italy.}\\
\textit{E-mail address:} {\fontfamily{cmtt}\selectfont
olivia.caramello@uninsubria.it
}\\
\\
\textsc{Istituto Grothendieck ETS, Corso Statuto 24, 12084 Mondovì, Italy.}
\textit{E-mail address:} {\fontfamily{cmtt}\selectfont olivia.caramello@igrothendieck.org}\\
\\
\textsc{Université Paris-Saclay, CentraleSupélec, Mathématiques et Informatique pour la Complexité et les Systèmes, 91190, Gif-sur-Yvette, France.}\\ 
\textit{E-mail address:} {\fontfamily{cmtt}\selectfont olivia.caramello@centralesupelec.fr}\\
\\
\\
\textsc{Yorgo Chamoun \\
\\
École Polytechnique, Laboratoire d’informatique (LIX), Bâtiment Alan Turing, 1 rue Honoré d’Estienne d’Orves, 91120, Palaiseau, France.}\\
\textit{Email address:} {\fontfamily{cmtt}\selectfont yorgo.chamoun@polytechnique.edu}

\end{document}